\documentclass[11pt]{article}
\usepackage{amsthm,amsfonts,amsmath,amssymb}
\usepackage{mathtools}
\usepackage{enumitem}
\usepackage{listings}
\usepackage{mathrsfs}
\usepackage{fancybox}
\usepackage{dsfont}
\usepackage{float}
\usepackage{arcs}
\usepackage{pgf,tikz}
\usetikzlibrary{arrows}
\usepackage{graphicx}
\usepackage{color}
\usepackage[all]{xy}

\usepackage{tikz}
   \usetikzlibrary{calc}

\newcommand{\overbow}[1]{
   \tikz [baseline = (N.base), every node/.style={}] {
      \node [inner sep = 0pt] (N) {$#1$};
      \draw [line width = 0.4pt] plot [smooth, tension=1.3] coordinates {
         ($(N.north west) + (0.1ex,0)$)
         ($(N.north)      + (0,0.5ex)$)
         ($(N.north east) + (0,0)$)
      };
   }
}

\newcommand{\Xd}{X_\Delta}
\newcommand{\Li}{\mathcal{L}}
\newcommand{\Lid}{\Li_\Delta}
\newcommand{\LL}{\lvert \Li \rvert}
\newcommand{\LLd}{\lvert \Lid \rvert}

\newcommand{\C}{\mathbb{C}}
\newcommand{\R}{\mathbb{R}}
\newcommand{\Z}{\mathbb{Z}}

\newcommand{\cp}[1]{\C P^{#1}}
\newcommand{\rp}[1]{\R P^{#1}}

\newcommand{\length}{\ell_\Z}
\newcommand{\aC}{\mathscr{C}}

\newcommand{\w}{w}

\newcommand{\rc}{C}
\newcommand{\nc}[1]{\mathscr{C}_{\geqslant #1}}

\newcommand{\Xb}{X^\bullet}

\newcommand{\cP}{\mathcal{P}}
\newcommand{\Cb}{C^\bullet}

\definecolor{bluegray}{rgb}{0.4, 0.6, 0.8}
\renewcommand{\emph}[1]{{{\color{bluegray} #1}}}

\newcommand{\A}{\mathcal{A}}
\newcommand{\ttor}{(\C^\ast)^2}

\newcommand{\sv}[1]{V_{#1}}

\newcommand{\edge}{\epsilon}
\newcommand{\D}{\mathcal D}
\newcommand{\mD}{\mathcal{D}}
\newcommand{\mR}{\mathcal{R}}
\newcommand{\mP}{\mathcal{P}}

\newcommand{\obs}{\varPsi}
\newcommand{\itrz}{\itr_\Z}
\newcommand{\mud}{\mu_\Delta}
\newcommand{\mudd}{\mu_{\Delta,1}}
\newcommand{\mul}{\mu_\Li}

\DeclareMathOperator{\ord}{\textit{ord}}

\DeclareMathOperator{\im}{im}
\DeclareMathOperator{\id}{id}

\DeclareMathOperator{\itr}{int}
\DeclareMathOperator{\conv}{conv}

\DeclareMathOperator{\PGL}{PGL}
\DeclareMathOperator{\Aut}{Aut}
\DeclareMathOperator{\dist}{d}
\DeclareMathOperator{\Sym}{Sym}
\DeclareMathOperator{\Hom}{Hom}
\DeclareMathOperator{\pic}{Pic}

\DeclareMathOperator{\conj}{conj}
\DeclareMathOperator{\Hess}{Hess}

\DeclareMathOperator{\area}{area}

\newtheorem{Proposition}{Proposition}[section]

\newtheorem{Definition}[Proposition]{Definition}
\newtheorem{Lemma}[Proposition]{Lemma}
\newtheorem{Remark}[Proposition]{Remark}
\newtheorem{Theorem}{Theorem}
\newtheorem{Conjecture}[Theorem]{Conjecture}
\newtheorem{Corollary}[Proposition]{Corollary}

\newtheorem*{ack}{Acknowledgement}

\date{}

\begin{document}

\title{Monodromy of rational curves on toric surfaces}
\author{Lionel Lang}

\maketitle

\footnote{{\bf MSC-2010 Classification:} 14D05, 14Q05.}

\abstract{For an ample line bundle $\Li$ on a complete toric surface $X$, we consider the subset $\sv{\Li} \subset \LL$ of irreducible, nodal, rational curves contained in the smooth locus of $X$. We study the monodromy map from the fundamental group of $\sv{\Li}$ to the permutation group on the set of nodes of a reference curve $C \in \sv{\Li}$.
We identify a certain obstruction map $\obs_{X}$ defined on the set of nodes of $C$ and show that the image of the monodromy is exactly the group of deck transformations of $\obs_{X}$, provided that $\Li$ is sufficiently big (in the sense we make precise below). Along the way, we construct a handy tool to compute the image of the monodromy for any pair $(X,\Li)$. Eventually, we present a family of pairs $(X, \Li)$ with small $\Li$ and for which the image of the monodromy is strictly smaller than expected.
}

\section{Introduction}\label{sec:intro}

For an ample line bundle $\Li$ on a complete toric surface $X$, we consider the variety $\sv{\Li}\subset\LL$ of irreducible nodal rational curves contained in the smooth locus of $X$. For a general curve $C \in \sv{\Li}$, any loop in $\sv{\Li}$ based at $C$ induces a permutation on the set of nodes of the curve $C$. This action is recorded by the monodromy map 
$$\mul : \pi_1\big( \sv{\Li}, C \big) \rightarrow \Aut \big( \big\lbrace \text{nodes of } C \big\rbrace \big).$$

The monodromy map $\mul$ plays an important role in different contexts. First, the image of $\mul$ can be thought of as a first approximation of the fundamental group $\pi_1\big( \sv{\Li}, C \big)$. From this perspective, the study of the map $\mul$ is in line with the work \cite{DL} on fundamental groups of complement to discriminant varieties, see \cite{CL1}, \cite{CL2} and \cite{Salt17} for recent developments. The study of $\mul$ also contributes to the Galois theory of enumerative problems in algebraic geometry. We refer to \cite{H79} and \cite{Vak06} for foundations and to \cite{SW}, \cite{E18} and \cite{EL} for recent developements. At last, the works \cite{Ha}, \cite{Tyom} and \cite{Tyom14} illustrate how the Severi Problem on toric surfaces can be decomposed into the study of $\mul$ and a separate problem in deformation theory.

In relation to the latter problem,  J. Harris proved in \cite{Ha} that the image of $\mul$ is the full permutation group when $(X, \Li)=\big(\cp{2},\mathcal{O}(d)\big)$. In \cite{Tyom}, I. Tyomkin proved the surjectivity of a slightly different monodromy map in the case of Hirzebruch surfaces. Apart from these two cases, it seems that  the image of $\mul$ is not known for general toric surfaces. 

In this paper, we aim to determine the image of $\mul$ for pairs $(X,\Li)$ as general as possible.
To start with, we investigate possible obstructions to permute the nodes of a curve $C\in \sv{\Li}$.
One obstruction is of purely topological nature and can be described as follows. Let $\Xb \simeq \ttor$ be the $2$-dimensional torus orbit in $X$ and denote $\Cb:=C\cap \Xb$ for any $C\in \sv{\Li}$. Then, the inclusion $\Cb \hookrightarrow \Xb$ induces a map $\pi:H_1(\Cb,\Z) \rightarrow H_1(\Xb,\Z)$. If we denote by $\cP \subset H_1(\Cb,\Z)$ the subspace generated by the punctures of $\Cb$, any node $\nu$ of $C$ defines a class $[\gamma_\nu] \in H_1(\Cb,\Z) \big/ \cP$ unique up to orientation. This is done by choosing a simple closed curve $\gamma_\nu \subset \Cb$ whose preimage in the normalisation of $C$ is a path joining the two branches of $\nu$. In turn, this defines the obstruction map 
$$\obs_X : \{\text{nodes of } C\} \rightarrow \big(H_1(\Xb,\Z)\big/ \pi(\cP)\big) \big/ \pm \id$$
sending a node $\nu$ to the projection of $[\gamma_\nu]$. The target space  of $\obs_X$ is essentially independent of $C$ and is determined by $X$, see Section \ref{sec:obs}. We will therefore denote it by $Q_X$. As a consequence, a loop in $\pi_1\big( \sv{\Li}, C \big)$ has to preserve the decoration of the nodes of $C$ by elements in $Q_X$.

Our main theorem states that the map $\obs_{X}$ is the only general obstruction. In order to formulate this result, we define for any constant $\ell\geqslant 1$, the set $\nc{\ell}(X)$ of ample line bundles $\Li$ on $X$ satisfying $\deg\big(\Li_{\vert C}\big)\geqslant \ell$ for any curve $C\subset X$.

\begin{Theorem}\label{thm:main}
For any  complete toric surface $X$, there exists a constant $\ell := \ell(X) \geqslant 1$ such that for any line bundle $\Li\in \nc{\ell}(X)$, the image of the monodromy map $\mul$  is the group of deck transformations of the map $\obs_X$.
\end{Theorem}

Theorem \ref{thm:main} is a consequence of Theorems \ref{thm:patch} and \ref{thm:combinatorics} proven in Sections \ref{sec:patch} and  \ref{sec:combinatorics}. We provide explicit constants $\ell(X)$ in Proposition \ref{prop:l}. However, we do not address the question of minimality of these constants in this paper. The reader interested in a particular pair $(X, \Li)$ not satisfying the hypotheses of Proposition \ref{prop:l} is referred to Theorem \ref{thm:patch}. 

Observe that the obstruction map $\obs_{X}$ already appeared implicitly in \cite[\S 4.1]{Tyom14}. We will see in Section \ref{sec:obs} that the target $Q_X$ of $\obs_{X}$ is reduced to a point whenever the toric surface $X$ is smooth. 
In particular, the map $\obs_{X}$ does not provide any obstruction in that case.

\begin{Conjecture}
For any smooth complete toric surface $X$, the image of the monodromy map $\mul$  is the full permutation group on the set of nodes of $C$.
\end{Conjecture}

The latter is motivated by computations using the present methods and other methods involving tropical geometry. With the above conjecture, we wish to emphasise that there is a substantial room for improvement in Theorem \ref{thm:main}. The latter theorem and the nature of the map $\obs_X$ imply that the surjectivity of the map $\mul$ depends only on the toric surface $X$, at least for sufficiently big $\Li$. In Section \ref{sec:counterex}, we show that the general situation is slightly more subtle. We provide examples of pairs $(X, \Li)$ for which $\obs_X$ is trivial and $\mul$ is however not surjective. These examples indicate the existence of other types of obstructions that would be worth investigating. They also indicate that the minimal constant $\ell(X)$ for which the conclusion of Theorem \ref{thm:main} holds is a non-trivial function of $X$.

Concerning applications, the existence of pairs $(X,\Li)$ with non-transitive monodromy map $\mul$ suggests possible generalisations of the example \cite{Tyom14} of toric surfaces with reducible Severi varieties. Indeed, it was proven in \cite{Ha} that the number of irreducible components of the Severi variety $\sv{d,g}$ of irreducible projective curves of genus $g$ and degree $d$ is bounded from above by the number of orbits of $g$-tuples of nodes of a curve $C \in \sv{d,0}$ under the monodromy action of $\sv{d,0}$. It is reasonable to expect, although it is highly non-trivial, that the same upper bound is valid for more general toric surfaces, see for instance \cite{Tyom}, \cite{Vak} and the forthcoming \cite{CHT20a}. In these circumstances, the Severi variety of genus $g$ curves associated to the pair $(X,\Li)$ is reducible only if the map $\mul$ is not $g$-transitive. Theorem \ref{thm:main} provides infinitely many such pairs. Following these observations, we show in \cite{LT} that such pairs $(X,\Li)$ carry reducible Severi varieties. 

Let us conclude this introduction with a brief description of our approach. The main idea is to use rational simple Harnack curves as an interface between the present monodromy problem and a purely combinatorial game. On the one hand, simple Harnack curves (see \cite{Mikh}) are real algebraic curves with remarkable properties (see \cite{MR}, \cite{KO}, \cite{PRis}) that are available in any ample linear system $\LL$. On the other hand, the pair $(X,\Li)$ corresponds to a lattice polygon $\Delta$ in the character lattice $M$ of the torus $\ttor \subset X$. Standard arguments on amoebas (see Section \ref{sec:harnack}) allow to identify the set of nodes of a rational simple Harnack curve $C\in \sv{\Li}$ with the set of points $\itr(\Delta)\cap M$. 
In turn, it allows us to express $\mul$ as a map $\mud : \pi_1\big( \sv{\Li}, C \big) \rightarrow \Aut\big( \itr(\Delta) \cap M\big)$ with corresponding obstruction map $\obs_{\partial\Delta}$ that we describe in Section \ref{sec:obs}.

In Section \ref{sec:maintool}, we study the map $\mud$ in the case when $\Delta:=T$ is a triangle. More precisely, we determine the image $G_T$ of the ``restriction'' of $\mu_T$ to a certain subspace of the Severi variety $\sv{T}$. 

In Section \ref{sec:patch}, we use methods inspired by Viro's Patchworking technique (see \cite{Viro08}) to show that the inclusion of a triangle $T$ in a general lattice polygon $\Delta$ induces the existence of a subgroup $\widetilde G_T$ of $\im(\mud)$ nearly equal to the group $G_T$ above.

In Section \ref{sec:combinatorics}, we show that under the assumptions of Theorem \ref{thm:main}, the groups $\widetilde G_T$, for all appropriate inclusions $T\subset \Delta$, generate the group of deck transformations of the map $\obs_{X}$.

\begin{ack} The author would like to extend his warmest thanks to Erwan Brugall\'{e}, R\'{e}mi Cr\'{e}tois, Alexander Esterov, Nick Salter, Kristin Shaw and Ilya Tyomkin for helpful discussions. The author is also grateful to an anonymous referee for numerous valuable comments on a prior version of this text.
\end{ack}

\tableofcontents

\section{Setting}\label{sec:settings}
In this section, we introduce the definitions and some elementary properties of objects mostly coming from toric geometry. 

We define complete toric surfaces starting from a lattice  $\emph{M} \simeq \Z^2$. Denote by $\emph{\Xb}$ the algebraic group $\Hom_{\Z}\big(M,\C^\ast\big) \simeq\ttor$. Reciprocally, the lattice $M$ is the lattice of character (or Laurent monomials) of the torus $\Xb$. A lattice polygon $\emph{\Delta}$ in $M_\R:= M \otimes \R$ is a $2$ dimensional polygon obtained as the convex hull of a finite set in $M$. We denote the set of interior lattice points $\emph{\itrz(\Delta)}:=\itr(\Delta) \cap M$ and $\emph{\length(v)}$ the integer length of any lattice vector $v \in M$. For technical reasons, we will always assume that the origin $0\in M$ is a vertex of $\Delta$.
Define the toric surface $\emph{X}:=X_\Delta \supset \Xb$ as the closure of the monomial embedding 
\begin{equation}\label{eq:monomialembed}
\Xb \hookrightarrow \cp{\vert \Delta \cap M\vert-1}
\end{equation}
given coordinate-wise by the monomials $\Delta \cap M$, see \cite[\S 2.3]{CLS}. A lattice polygon defines a toric surface together with a line bundle. For a more intrinsic definition, we refer to the construction using fans, see \cite[\S 1.4]{Ful}.

Throughout the text, we denote by $\emph{n}$ the number of edges of the polygon $\Delta$ defining the toric surface $X$ and fix once and for all a counter-clockwise cyclical indexation of these edges in $\Z/n\Z$. For any $j\in \Z/ n\Z$, we denote by $\emph{\Delta_j}$ the $j^{th}$ edge of $\Delta$ and by $\emph{\Delta_{j,j+1}}$ the vertex $\Delta_{j}\cap \Delta_{j+1}$. We will also denote by $\emph{v_j} \in M$ the primitive integer vector supporting $\Delta_j$ and agreeing with the counter-clockwise orientation on $\partial \Delta$.

The group action of $\Xb$ on itself extends to whole of $X$. To each edge $\Delta_j$ corresponds the closure of an $\Xb$-orbit of dimension $1$ in $X$ that we refer to as a \emph{toric divisor} of $X$. For any $j\in \Z/ n\Z$, we denote by $\emph{\D_j} \simeq \cp{1}$ the orbit corresponding to $\Delta_j$. The classes of the divisors $\D_j$, $j\in \Z/ n\Z$, generate the Picard group $\pic(X)\simeq H^2(X,\Z)$, see \cite[\S 3.4]{Ful}. In particular, any divisor class $[D]\in \pic(X)$ is determined by its intersection multiplicities with the $\D_j$-s.

On the toric surface $X$, a line bundle $\emph{\Li}$ is ample if and only if it is very ample if and only if $\emph{\ell_j}:=\deg\big(\Li_{\vert \D_j} \big) >0$ for any $j\in \Z/ n\Z$, see \cite[Theorems 6.3.13 and 6.1.14]{CLS}. As mentioned above, the integers $\ell_j$, $j\in \Z/ n\Z$, determine the line bundle $\Li$. Moreover, they satisfy
\begin{equation}\label{eq:linebundle}
\sum_{j\in \Z/ n\Z} \ell_j \cdot v_j=0
\end{equation}
Conversely, every collection $\big\lbrace\ell_j\big\rbrace_{j\in \Z/ n\Z}$ satisfying \eqref{eq:linebundle} is given by $\ell_j=\deg\big(\Li_{\vert \D_j} \big)$ for some line bundle $\Li$ on $X$, see \cite[Proposition 6.4.1]{CLS}. Equivalently, an ample line bundle on $X$ is encoded by the lattice polygon obtained as the concatenation of the vectors $\ell_1 \cdot v_1,\cdots, \ell_n \cdot v_n$ for integers $\ell_j\geqslant 1$ satisfying \eqref{eq:linebundle}. For such lattice polygon $\Delta$, we will denote by $\emph{\Li_\Delta} \in \aC(X)$ the corresponding element in the set of ample line bundles $\emph{\aC(X)}$ on $X$. We also define
\[\emph{\aC_{\geqslant \ell}(X)}:= \left\lbrace \Li \in \aC_{\geqslant \ell}(X) \; \big| \; \deg\big(\Li_{\vert \D_j} \big) \geqslant \ell, \; j\in \Z/ n\Z\right\rbrace.\]

For any ample line bundle $\Li$ on $X$ and any point $C \in \LL$, we abusively denote by $C$ the corresponding curve $C\subset X$. We denote by $\emph{\sv{\Li}}\subset \LL$ the space of irreducible nodal rational curves contained in the smooth locus of $X$. For simplicity, we also denote $\emph{\sv{\Delta}}:=\sv{\Li_\Delta}$. According to \cite[Proposition 4.1]{Tyom}, the set $\sv{\Li}$ is non-empty and irreducible. Moreover, each curve $C \in \LL$ can be parametrised explicitly as follows, see \cite[Section $4$]{Tyom}. Choose coordinates $\Xb\simeq \ttor$ such that we can write $v_j=(\beta_j,-\alpha_j)$ in the induced coordinates on $M$. If $\Li$ is given by the collection  $\big\lbrace\ell_j\big\rbrace_{j\in \Z/ n\Z}$, then any irreducible rational curve $C \in \LL$ admits a parametrisation of the form
\begin{eqnarray}\label{eq:param}
\begin{array}{rcl}
 \phi \; : \; \cp{1} 	& \dashrightarrow	& \ttor \\
 t  							& \mapsto 			& \displaystyle \Big( z_0 \,  \prod_{j\in \Z/ n\Z} \; \;  \prod_{1\leq l \leq \ell_j} \big(t-a_{j,l}\big)^{\alpha_j}, \; w_0 \,  \prod_{j\in \Z/ n\Z} \; \; \prod_{1\leq l \leq \ell_j} \big(t-a_{j,l}\big)^{\beta_j}\Big)	
\end{array}
\end{eqnarray}
where  $z_0, \, w_0 \in \C^\ast$ and  $a_{j,l} \in \cp{1}=\C \cup\{\infty\}$. In the above formula, any factor $t-a_{j,l}$ with $a_{j,l}=\infty$ is to be replaced with the constant factor $1$. Such a representation is unique up to the action of $\PGL_2(\C)$ on the parameter $t$. 

Conversely, for an irreducible curve $C$ parametrised by \eqref{eq:param} to be in $\LL$, it suffices that the $a_{j,l}$-s are pairwise distinct. In this situation, the curve $C$ avoids the set of torus fixed point of $X$ which is a superset of the singular locus of $X$. Observe also that the curve $C$ is nodal for a generic choice of parameters $a_{j,l}$. This fact can be justified for instance by the existence of rational simple Harnack curve in the linear system $\LL$, see Section \ref{sec:harnack}. It follows then from  \cite{Kho} that $C$ has exactly $\vert \itrz(\Delta)\vert$ many nodes. As a consequence, the variety $\sv{\Li}$ can be seen as an open Zariski subset of the space of parameters $(z_0,w_0)$, $\left\lbrace a_{j,l}\right\rbrace$, quotiented by the action of $\PGL_2(\C)$.  In particular, the set $\sv{\Li}$ has dimension $\vert \partial \Delta \cap M \vert-1$. 

Let us now define the monodromy map $\mu_{\Li}$. Below, we use $\emph{\Aut(\mathcal{E})}$ to denote the group of permutations on a finite set $\mathcal{E}$ and $\emph{\Aut(f)}$ to denote the group of deck transformations of a map $f$ between finite sets. For a reference curve $C \in \sv{\Li}$, we denote 
\begin{equation}\label{eq:monomap}
\emph{\mu_\Li} : \pi_1 \big( \sv{\Li}, C \big) \rightarrow \Aut \big( \big\lbrace \text{nodes of } C \big\rbrace \big)
\end{equation}
the \emph{monodromy map} of the covering $\left\lbrace (C,\nu)\in \sv{\Li}\times X \big| \nu \in C_\theta \text{ is a node}  \right\rbrace \rightarrow \sv{\Li}$ formally defined as follows. For a loop $\gamma:=\big\lbrace C_\theta \big\rbrace_{\theta \in [0,1] } \subset \sv{\Li}$ based at $C$, that is $C=C_0=C_1$, consider the map 
\[ \Phi : \big\lbrace (\theta, \nu) \in [0,1] \times X \; \big| \; \nu \in C_\theta \text{ is a node}  \big\rbrace \rightarrow 
\big\lbrace \text{nodes of } C \big\rbrace \]
that associates to any point $(\theta,\nu)$ the node $\nu'\in \big\lbrace \text{nodes of } C \big\rbrace$ such that $(0,\nu')$ is the initial point of the connected component of the source space of $\Phi$ containing $(\theta,\nu)$. Since $C_1=C$, the map $\Phi(1, \, \_ \;)^{-1}$ is an element of $\Aut\big( \big\lbrace\text{nodes of } C \big\rbrace\big)$ that only depends on the homotopy class $[\gamma] \in \pi_1 \big( \sv{\Li}, C \big)$. Then, we define $\mu_\Li \big( [\gamma]\big) := \Phi(1, \, \_ \;)^{-1}$.

\section{Rational simple Harnack curves}\label{sec:harnack}

Rational simple Harnack curves will be a useful tool to describe the  obstruction map $\obs_X$ (see Section \ref{sec:obs}) and to construct explicit subgroups in $\im(\mul)$ (see Section \ref{sec:patch}). In this section, we recall some basic facts about these curves.

First, we shall fix coordinates $(z,w)$ on $\Xb \simeq \ttor$.
The complex conjugation on $\Xb \simeq \ttor$ extends to an anti-holomorphic involution $\emph{\conj}$ on $X$. A curve $C \in \LL$ is \emph{real} if $\conj(C)=C$ and we denote by $\emph{\R C}$ the fixed locus of $\conj_{\vert C}$. Recall the \emph{amoeba map}
\[
\begin{array}{rcl}
\emph{\A} : \ttor & \rightarrow & \R^2\\
(z,w) & \mapsto & \big( \log \vert z \vert, \log \vert w \vert \big)

\end{array}
\]
and the notation $\emph{\Cb}:=C\cap\Xb$ from the introduction.

\begin{Definition}
A real curve $C \in \LL$ is a (possibly singular) \emph{simple Harnack curve} if the restriction of the amoeba map $\A : \Cb \rightarrow \R^2$ is at most $2$-to-$1$.
\end{Definition}

The above definition is shown to be equivalent to the original definition \cite[Definitions $2$ and $3$]{MR} in \cite[Theorem $1$]{MR}. Recall that smooth simple Harnack curves are maximal, that is $b_0(\R C)=g+1$ where $g$ is the genus of $C$; only one component of $\R C$ intersects the toric divisors of $X$; the restriction of $\A$ to $\Cb$ realises the quotient of $\Cb$ by $\conj$. In particular, the boundary of $\A(\Cb)$ is exactly $\A(\R \Cb)$ (see \cite[Lemma 8]{Mikh}) and $\A(\Cb)$ is a closed subset of $\R^2$ with $g$ holes.
%
%

Singular simple Harnack curves are now obtained from smooth ones by contracting some of the  holes of the amoeba, or equivalently, by contracting some of the compact ovals of $\R \Cb$ to points. A local model for such contractions is given by $z^2+w^2=\varepsilon$, $0\leq \varepsilon <1$. In particular, the only singularities of simple Harnack curves are real isolated double points. We refer to \cite{MR} for more details.

The existence of smooth simple Harnack curves in $\LL$ is guaranteed by \cite[Corollary A4]{Mikh}. For singular curves, the existence is addressed in \cite[Theorem 6]{KO}, \cite[Theorems 2 and 10]{Bru14},  \cite[Theorem 3]{CL1} in various contexts. For rational curves, there is an explicit construction. Recall that any real rational curve $C \in \LL$ admits a parametrisation as in \eqref{eq:param} with $z_0, \, w_0 \in \R^\ast$ and  $a_{j,l} \in \rp{1}$. Fix an orientation on $\rp{1}$ so that the collection of parameters $a_{j,l} \in \rp{1}$ inherits a cyclical ordering. As a consequence of \cite[Theorem 10]{Bru14}, we have the following.

\begin{Proposition}\label{prop:ratharnack}
A real rational curve in $\sv{\Li}$  is a simple Harnack curve if and only if it can be parametrised as in \eqref{eq:param} with $z_0, \, w_0 \in \R^\ast$ and  $a_{j,l} \in \rp{1}$ such that for any $j\in \Z/ n\Z$ and $1\leq l \leq \ell_j$, we have $$a_{(j-1), \ell_{j-1}} < a_{j,l} < a_{(j+1),1}.$$
In particular, there always exist rational simple Harnack curves in the linear system $\LL$.
\end{Proposition}

Recall that for smooth simple Harnack curves $C \in \LLd$, the order map $\emph{\ord}$ of \cite{FPT} establishes a bijective correspondence between the set of compact connected components of $\R \Cb$ and the set of lattice points $\itrz(\Delta)$, see \cite[Corollary $10$]{Mikh}. We now aim to extend this correspondence to rational simple Harnack curves. To do so, we use the existence of deformations of rational simple Harnack curves to smooth simple Harnack curves. The existence of such deformations can be proven using the machinery of \cite[\S 4.5]{KO} in general toric surfaces, see \cite{Ola}. As a more direct approach, take $P(z,w)$ to be a real Laurent polynomial of Newton polygon $\Delta$ defining $C$. Then, we can find a real polynomial $R$ with Newton polygon $\Delta$ such that, for any node $\nu\in C$, we have $R(\nu)>0$ (respectively $R(\nu)<0$) if $\Hess P(\nu)$ is positive definite (respectively negative definite). A small deformation of $P$ in the direction of $R$ gives the desired smoothing.

According to \cite[Lemma 11]{Mikh}, the map $\ord$ on a smooth simple Harnack curve $C$ can be described as follows. First, assume that the vertex $\Delta_{n,1}$ is the origin. Let $c_0\subset\R \Cb$ be the unique connected component joining the two consecutive toric divisors $\D_{n}$ and $\D_{1}$. For any compact component $c$ of $\R \Cb$, draw a path $ \rho_c \subset \A(C)$ joining $\A(c_0)$ to $\A(c)$. By the $2$-to-$1$ property of the amoeba map, the lift $\gamma_c:=\A^{-1}(\rho_c)$ is a loop in $\Cb$ invariant by complex conjugation. There is exactly one orientation of the latter loop such that the corresponding homology class $(a,b) \in H_1(\ttor, \Z)$ satisfies $(-b,a) \in \itrz(\Delta)$ (note the sign mistake in the sixth line of the proof of \cite[Lemma 11]{Mikh}). Then, we have $\ord(c)=(-b,a)$.

For a rational simple Harnack curve $C \in \sv{\Li_\Delta}$ and a node $\nu \in C$, we can repeat the same construction where $\rho_\nu$ is a path joining $\A(c_0)$ to $\A(\nu)$ and $\gamma_\nu:=\A^{-1}(\rho_\nu)$.

\begin{Definition}\label{def:ord}
For a rational simple Harnack curve $C \in \sv{\Li_\Delta}$, a node $\nu\in C$, define the \emph{order map}
$$\ord(\nu)=(-b,a) \in \itrz(\Delta)$$ where $(a,b) \in H_1(\ttor, \Z)$ is the homology class of the (carefully oriented) loop $\gamma_\nu$ constructed above.
\end{Definition}

The fact that  $(-b,a) \in \itrz(\Delta)$ (with the appropriate orientation of $\A^{-1}(\gamma)$) follows from continuity by applying the above construction to a nearby smooth simple Harnack curve.

\begin{Proposition}\label{prop:extord}
For a rational simple Harnack curve $C \in \sv{\Li_\Delta}$, the map 
\[ \ord : \lbrace \text{nodes of } C\rbrace \rightarrow \itrz(\Delta)\]
is a bijection.
\end{Proposition}

\begin{proof} This is a direct consequence of the existence of smoothings and the fact that the order map on smooth simple Harnack curve is bijective, see \cite[Corollary $10$]{Mikh}.
\end{proof}

For an ample line bundle $\Li=\Li_\Delta$ on $X$ and a rational simple Harnack curve $C\in \sv{\Li}$, define the \emph{monodromy map} 
\begin{equation}\label{eq:monomapdelta}
\emph{\mud} : \pi_1 \big( \sv{\Li}, C \big) \rightarrow \Aut \big( \itrz(\Delta) \big).
\end{equation}
by  ``composition'' of the monodromy map $\mu_\Li$ of \eqref{eq:monomap} with the order map of Definition \ref{def:ord}. Formally, we have $\mud\big([\gamma]\big)= \ord \circ \Phi\big(1, \, \_\; \big) \circ \ord^{-1}$ where $\Phi$ is the trivialisation used to define $\mu_\Li$ in \eqref{eq:monomap}.

\section{Obstructions}\label{sec:obs}

In this section, we investigate the obstructions constraining the image of the monodromy map $\mu_\Li$. 

As before, denote by $\Xb\subset X$ the 2-dimensional torus orbit and denote $\emph{N}:= H_1(\Xb,\Z)$. Any toric divisor $\D_j\subset X$, $j \in \Z/n\Z$, defines the 1-dimensional sublattice $N_j\subset N$ generated by the class of the boundary of a small disc $D\subset X$ intersecting $\D_j$ transversally away from a torus fixed point. Define $N_X \subset N$ as the sum of $N_j$ over all $j \in \Z/n\Z$. 

For any curve $C\in \sv{\Li}$, the class of a small loop around any of the punctures of $\Cb$ is an element in $N_X$. Indeed, the latter claim is clear if $C$ intersects transversally each of the toric divisors. Since the space of such curves is an open Zariski subset of $\sv{\Li}$, the claim is also true in general.

In turn, we can associate to any node $\nu \in C$ a path in the normalisation of $\Cb$ that joins the two branches of $\nu$. This path maps to a loop $\gamma_\nu \subset \Cb$. Once an orientation is fixed on $\gamma_\nu$, or equivalently once an ordering of the two branches of $\nu$ is fixed, the projection of the homology class  $[\gamma_\nu]$ in $N\big/ N_X$ does not depend on the choice of the above path, since the normalisation of $\Cb$ is a punctured sphere. Since switching the ordering of the two branches of $\nu$ exchanges $[\gamma_\nu]$ with $-[\gamma_\nu]$, the map
\[ \emph{\obs_{X,C}} : \{\text{nodes of } C\} \rightarrow \big(N\big/ N_X\big)\big/ \pm \id\]
sending the node $\nu$ to the projection of the class $[\gamma_\nu]$, for any ordering of the branches of $\nu$, is independent of any choice. 

\begin{Lemma}\label{lem:obsmap}
The image of the monodromy map $\mul : \pi_1\big( \sv{\Li}, C \big) \rightarrow \Aut \big(\lbrace \text{nodes of } C \rbrace \big)$ is a subgroup of the group of deck transformations $\Aut(\obs_{X,C})$.
\end{Lemma}

\begin{proof}
Let $\gamma : S^1 \rightarrow \sv{\Li}$ be a loop based at $C$ and denote $C_\theta:=\gamma(\theta)$. For any $\theta \in S^1$, the map $\obs_{X,C_\theta}$ defined on the curve $C_\theta$ is continuous in $\theta$. It follows that the permutation on the nodes of $C$ induced by the loop $\gamma$ has to preserve the fiber of the map $\obs_{X,C}$.
\end{proof}

From now on, we will simply denote by $\emph{\obs_X}$ the map $\obs_{X,C}$ and by $\emph{Q_X}$ the target space of $\obs_X$.
We now describe the obstruction map $\obs_X$ in terms of the order map given in Definition \ref{def:ord}. Let $\Delta \subset M_\R$ be the lattice polygon corresponding to the line bundle $\Li$. Recall that, by convention, the origin $0\in M$ is a vertex of $\Delta$, see Section \ref{sec:settings}. Define $\emph{M_\Delta} \subset M$ to be the lattice of finite index generated by $\partial \Delta \cap M$ and define $\emph{Q_\Delta}:= (M / M_\Delta)\big/ \pm \id$. Finally, define \emph{the obstruction map} 
$$ \obs_{\partial\Delta} : \itrz(\Delta) \rightarrow Q_\Delta$$
to be the restriction of the quotient map $M \rightarrow Q_\Delta$.

\begin{Proposition}\label{prop:obs}
Let $\Li \in \aC(X)$ and $\Delta\subset M_\R$ such that $\Li=\Li_\Delta$.
For any simple Harnack curve $C \in \sv{\Li}$, the maps $\obs_X$ and $\obs_{\partial\Delta} \circ \ord$ defined on $\lbrace \text{nodes of } C \rbrace$ have identical fibers. In particular, the image $\mud : \pi_1 \big( \sv{\Li}, C \big) \rightarrow \Aut \big(\itr(\Delta) \cap \Z^2\big)$  is a subgroup of 
$\Aut(\obs_{\partial\Delta})$.
\end{Proposition} 

\begin{proof}
As in Section \ref{sec:harnack}, let us fix coordinates $(z,w)$ on  $\Xb \simeq \ttor$ so that both $M$ and $N$ inherits coordinates $M\simeq \Z^2$ and $N\simeq \Z^2$.

On the one hand, it follows from standard computations, using for instance the description \eqref{eq:monomialembed} of $X$, that the sublattice $N_j \subset N$ is supported by  $(-b,a)$ where $(a,b)$ is the primitive integer vector directing the edge $\Delta_j\subset \Delta$. In particular, the map $\Z^2\rightarrow \Z^2$, $(a,b) \mapsto (-b,a)$, maps $M_\Delta$ to $N_X$ so that $Q_\Delta\simeq Q_X$.

On the other hand, for any node $\nu\in C$, we can choose the loop $\gamma_\nu\subset C$ used to defined $\obs_X(\nu)$ as we did in Section \ref{sec:harnack} to define $ord(\nu)$. In particular, $ord(\nu)$ is the projection of the vector $(-b,a)$ in $Q_\Delta$ if $\obs_X(\nu)$ is the projection of the vector $(a,b)$ in $Q_X$. The result follows.
\end{proof}

\begin{Corollary}\label{cor:nonsurj}
The map $\mud$ is not surjective whenever $[M:M_\Delta]\geqslant 4$ or $[M:M_\Delta] \in \{2,3\}$ and $M_\Delta \cap \itr(\Delta)$ is non-empty.
\end{Corollary}

\begin{proof}
Let $\alpha:=[M:M_\Delta]$ with $\alpha \geqslant 2$. Then, there are exactly $\lfloor \alpha/2 \rfloor \geqslant 1$ classes in $Q_\Delta$ that are distinct from the class of the lattice $M_\Delta$. In order to prove the statement, it suffices to show that there exists a lattice point in $\itrz(\Delta)$ that projects to any such class. Indeed, the map $\obs_{\partial\Delta}$ will have strictly more than one fiber under the assumption of the statement. As the monodromy $\mud$ has to preserve these fibers by Proposition \ref{prop:obs}, the map $\mud$ cannot be surjective.

First, we claim that there exists a lattice triangle $T\subset \Delta$ such that $T \cap \partial \Delta \cap M = \lbrace \text{vertices of } T \rbrace$. To see this, consider a lattice triangle $T\subset \Delta$ obtained as the convex hull of three consecutive lattice points on $\partial \Delta$. If $T \neq \Delta$, then we are done. If $\Delta=T$ and $T$ does not satisfies the assumption of the claim, the convex hull of any other triple of consecutive and non-collinear lattice points on $\partial \Delta$ will. The claim follows.

Take now $T\subset \Delta$ to be a lattice triangle such that $T \cap \partial \Delta \cap M= \lbrace \text{vertices of } T \rbrace$. By construction, the vertices of $T$ are in $M_\Delta$. We claim that $\obs_{\partial\Delta} \big( (T\cap M) \setminus \lbrace \text{vertices of } T \rbrace\big)$ always contains $Q_\Delta \setminus 0$. To see this, take coordinates on $M$ such that $(0,0)$, $(1,0)$ and $(p, \alpha q)$ are the vertices of $T$ where $m, \, p,  \, q \in \Z_{>0}$. In particular, we have $M_\Delta= \Z \oplus \alpha \Z$ and $\obs_{\partial\Delta}(u,v)=\dist(v, \alpha \Z)$ when identifying $Q_\Delta \simeq \lbrace 0, 1, ..., \lfloor \alpha/2 \rfloor \rbrace$ (here, \emph{$\dist$} is the Euclidean distance in $\R^2$). The triangle $T$ contains exactly one lattice point $p_h$ on the union of the two horizontal lines of respective heights $h$ and $\alpha q-h$, for any integer $0<h<\alpha q$. Indeed, there is exactly one lattice point at height $h$ in the interior of the parallelogram $\conv\big(T,(p, \alpha q+1) \big)$. Since $\obs_{\partial\Delta}$ maps the point $p_h$ to $\dist(h, \alpha \Z)$ for any integer $0 < h < \alpha q$, the claim follows.
\end{proof}

\begin{Remark}
If $p_j \in X$ is the toric fixed point corresponding to the vertex $\Delta_{j,j+1}$, then $p_j$ is a singular point of $X$ if and only if $m_{j}:=\big[N: \left\langle v_j, \, v_{j+1} \right\rangle\big] \geqslant 2$. In such case, the point $p_j$ is a cyclic quotient singularity of $X$ of type $\C^2 \big/ (\Z/m_j \Z)$, see \cite[Chapter 2, \S 2.2]{Ful}. From there, this is not hard to see that $[M:M_\Delta]= \gcd\big(\{m_j\}_{j \in \Z / n\Z}\big)$. In particular, we have $\vert Q_\Delta \vert = \lfloor \gcd(\{m_j\}) /2 \rfloor +1$.
\end{Remark}

\section{Monodromy in weighted projective planes}\label{sec:maintool}

In this section, we focus on the simplest kind of polygons, namely triangles. We assume that the lattice polygon $\Delta \subset M_\R$ is a triangle such that $\length(\Delta_1) \geqslant 2$. In particular, the associated toric surface $\Xd$ is a weighted projective plane. We define the subset $\emph{\sv{\Delta,1}} \subset \sv{\Delta} \subset \LLd$ to be the space of curves intersecting the toric divisors $\D_2, \; \D_3 \subset \Xd$ only once. Since $\sv{\Delta,1} \subset \sv{\Delta}$, the points of intersection with $\D_2$ and $\D_3$ are necessarily disjoint from the three torus fixed points of $\Xd$.

Recall that if $C \in \sv{\Delta}$ is a rational simple Harnack curve, we can consider the monodromy map $\mud : \pi_1 \big(\sv{\Delta}, \rc \big) \rightarrow \Aut \big(\itrz(\Delta)\big)$.  According to Proposition \ref{prop:ratharnack}, the simple Harnack curve $\rc$ can be taken in $\sv{\Delta,1}$ by simply requiring that all the parameters $a_{2,l}$ (respectively $a_{3,l}$) are equal to each other. In this case, we can consider the map $\pi_1( \sv{\Delta,1}, \rc \big) \rightarrow \pi_1 \big(\sv{\Delta}, \rc \big)$ induced by the inclusion $\sv{\Delta,1} \hookrightarrow \sv{\Delta}$ and consider the composition of the latter with $\mud$ that we denote $\emph{\mudd} : \pi_1 \big(\sv{\Delta,1}, \rc \big) \rightarrow \Aut \big(\itrz(\Delta)\big)$.

Our interest in the subspace $\sv{\Delta,1} \subset \sv{\Delta}$ is motivated by the following. First, we will show in Theorem \ref{thm:triangle} that we can determine the image of $\mudd$ by explicit computation. Second, we will see in Section \ref{sec:patch} that this computation allows us to construct elements in the image of $\mud$ for general $\Delta$ via a patchworking procedure.

In order to state Theorem \ref{thm:triangle}, we need to introduce the relevant obstructions. Let $\emph{M_{\Delta,1}} \subset M$ be the affine lattice generated by $\Delta_1 \cap \Z^2$ and the vertex $\Delta_{2,3}$. Since $0\in M$ is a vertex of $\Delta$ by convention, the subset $M_{\Delta,1}$ is an honest lattice. Thus, we define $\emph{Q_{\Delta,1}}:=\big(M \big/ M_{\Delta,1}\big)\big/ \pm \id$. Finally, define the \emph{obstruction map} 
$$ \emph{\obs_{\Delta,1}} : \itrz(\Delta) \rightarrow Q_{\Delta,1}$$
to be the restriction of the quotient map $M \rightarrow Q_{\Delta,1}$. 

\begin{Theorem}\label{thm:triangle}
The image of $\mudd$ is the group of deck transformations of $ \obs_{\Delta,1}$.
\end{Theorem}

In order to prove Theorem \ref{thm:triangle}, we fix a system of coordinates on $M$ such that the 
edge $\Delta_1$ is the segment joining $(0,0)$ to $(\ell,0)$ where $\ell := \length(\Delta_1)$ and 
such that the vertex $\Delta_2 \cap \Delta_3$ has coordinates $(p,q)$ for some $p, \, q \in 
\Z_{>0}$. We assume that $q\geqslant 2$, otherwise $\itrz(\Delta)= \varnothing$ and there is 
nothing to prove. Thus, we have that $M_{\Delta,1} = \Z \oplus q\Z$ and we can identify 
$Q_{\Delta,1} \simeq \big\lbrace 0,1,...,\lfloor q/2 \rfloor\big\rbrace$ and $\obs_{\Delta,1}(u,v)= 
\dist(v, q\Z)$, where we recall that $\dist$ is the Euclidean distance. Note that $0 \notin \im(\obs_{\Delta,1})$ and that the cardinality of each fiber $
\obs_{\Delta,1}^{-1}(k)$ usually depends on $k \in \big\lbrace1,...,\lfloor q/2 \rfloor\big
\rbrace$. It follows that the group of deck transformations of $\obs_{\Delta,1}$ is 
\[\prod_{k=1}^{\lfloor q/2 \rfloor} \Aut \big( \left\lbrace(u,v)\in\itr(\Delta) \cap M \; \big| \; v \in \{k, \, q-k\}  \right\rbrace \big).\]

To begin with, we describe the double points of any simple Harnack curve $C \in \sv{\Delta,1}$ within each fiber of the obstruction map $\obs_{\Delta,1}$. To that aim, we first observe in Lemma \ref{lem:param3} that the parametrisation \eqref{eq:param} admits a very simple form in the present setting. Using this, we then describe explicitly the elements $t\in \cp{1}$ parametrising the nodes of $C$ in Proposition \ref{prop:description}.

\begin{Lemma}\label{lem:param3}
Up to a toric translation in $\Xd$, any curve $C \in \sv{\Delta,1}$ admits a parametrisation of the form
\begin{equation}\label{eq:param3}
\phi_{\{a\}} \; : \; t \mapsto \Big(t^q, \; t^{-p} \prod_{ j=1}^{\ell} (t-a_{j})\Big)
\end{equation}
where $\{a\}:=\big\lbrace a_1,...,a_\ell\big\rbrace \in \Sym_\ell (\C^\ast)$.
\end{Lemma}

Observe that we made a slight abuse of notation above while using regular braces in $\big\lbrace a_1,...,a_\ell\big\rbrace$. In general, $\{a\}$ will be a multi-set. However, we will stick to this notation for the rest of the paper.

\begin{proof}
By definition of $\sv{\Delta,1}$, the two points $C \cap \D_{2}$ and $C \cap \D_{3}$ are distinct. Up to the action of $\PGL_2(\C)$ by pre-composition on the parametrisation \eqref{eq:param}, we can assume that $C \cap \D_{2}$ is parametrised by $\infty$ and  $C \cap \D_{3}$ by $0$. By definition of $\sv{\Delta,1}$ again, all remaining parameters $a_{j,l}$ of \eqref{eq:param} are distinct from $0$ and $\infty$, otherwise $C$ would pass through a toric fixed point. Then, the parametrisation \eqref{eq:param} gives
\[ t \mapsto \Big( z_0 \, t^q , \; w_0 \, t^{-p} \prod_{j=1}^\ell (t-a_j) \Big).\]
We recover the announced formula after translating by $(z_0^{-1}, w_0^{-1})$ in $\Xd$.
\end{proof}

Let us now introduce the polynomials in the variable $t$ that we will use to describe the nodes of $C$. For any integer $k \in \big\lbrace 1,...,\lfloor q/2 \rfloor\big\rbrace$ such that $q\neq2k$ and any  $\{a\}:= \{a_1,...,a_\ell\} \in\Sym_\ell (\C^\ast)$, define 
\begin{equation}\label{eq:pol}
 \emph{P_{k,\{a\}}(t)}:=  \sum_{j=0}^\ell (-1)^{\ell-j}  \sigma_{\ell-j}(a_1,...,a_\ell) \cdot \sin\Big(\frac{k\pi(p-j)}{q}\Big)  \cdot \big(t e^{\frac{i\pi k}{q}}\big)^j
\end{equation}
where $\sigma_j$ is the elementary symmetric polynomial of degree $j$ in $\ell$ variables. In the case $q=2k$, the support of the above polynomial is the intersection of the affine sublattice $(p+1)+2\Z$ with $\{0,\cdots,\ell\}$. Recall that the \emph{support} is the set of exponents of the monomials appearing in the polynomial. For technical reason, we define instead the polynomial 
\begin{equation}\label{eq:pol2}
 \emph{P_{k,\{a\}}(t)}:=  \sum_{j=0}^\ell (-1)^{\ell-j}  \sigma_{\ell-j}(a_1,...,a_\ell) \cdot \sin\Big(\frac{k\pi(p-j)}{q}\Big)  \cdot \big(t e^{\frac{i\pi k}{q}}\big)^{\varepsilon(j)}
\end{equation}
where $\varepsilon(j)=(j-1)/2$ if $p$ is even and $\varepsilon(j)=j/2$ otherwise. For any $k \in \big\lbrace 1,...,\lfloor q/2 \rfloor\big\rbrace$, let us denote by $\emph{J_k} \subset \left\lbrace 0,...,\ell \right\rbrace$ the support of the polynomial $P_{k,\{a\}}$. We will study these supports in Lemma \ref{lem:support}.  Finally, define the multiplier $\emph{\lambda_k}\in \C^\ell$ whose coordinates $\lambda_{k,j}$ are defined by
\[
P_{k,\{a\}}(t)=  \sum_{j=0}^\ell (-1)^{\ell-j} \sigma_{\ell-j}(a_1,...,a_\ell) \cdot \lambda_{k,j} \cdot t^j.
\]

\begin{Proposition}\label{prop:description}
Let $C \in \sv{\Delta,1}$ be parametrised by $\phi:=\phi_{\{a\}}$ as in \eqref{eq:param3}. Then, we have the following bijective correspondence
\[
\begin{array}{rcl}
 \textbf{P} \; : \; \displaystyle \bigsqcup_{1\leq k \leq \lfloor q/2 \rfloor} \big\lbrace  t \in \C^\ast \; \big| \; P_{k, \{a\}}(t)=0 \big\rbrace & \longrightarrow & \big\lbrace \text{nodes of }\;  C \big\rbrace\\
 (k,t)   & \longmapsto & 
\left\lbrace
\begin{array}{ll}
\nu=\phi\big(\sqrt{t}\big)=\phi\big(-\sqrt{t}\big) & \text{if } q = 2k\\
\nu=\phi(t)=\phi\big(te^{2i\pi k/q}\big) & \text{otherwise}
\end{array}
\right. 
\end{array}.
\]
If $C$ is a simple Harnack curve, then $\textbf{P}\big(\big\lbrace  t \in \C^\ast \; \big| \; P_{k, \{a\}}(t)=0 \big\rbrace\big)=\big(\obs_{\Delta,1}\circ \ord\big)^{-1}(k)$. In particular, we have $\left| \obs_{\Delta,1}^{-1}(k)\right| = \left| \big\lbrace  t \in \C^\ast \; \big| \; P_{k, \{a\}}(t)=0 \big\rbrace \right| $.
\end{Proposition}

\begin{Corollary}\label{cor:obs}
The image of $\mudd$  is a subgroup of  $ \Aut(\obs_{\Delta,1})$.
\end{Corollary}

\begin{proof}
The proof is similar to the proof of Lemma \ref{lem:obsmap}.
\end{proof}

\begin{proof}[Proof of Proposition \ref{prop:description}]
By assumption, the curve $C$ has only nodes as singularities. Any node $\nu \in C$ corresponds to an unordered pair $\{t, \, t'\} \subset \C^\ast$ such that $\phi(t)=\phi(t')$. By \eqref{eq:param3}, the latter equality is equivalent to 
\[ \big( t \big/ t'\big)^q= 1 \;  \text{ and } \; \prod_{j=1}^\ell\dfrac{t-a_j}{t'-a_j} = \big( t \big/ t'\big)^p.\]
In particular,  there exists a unique $k \in \big\lbrace 1,...\lfloor q/2\rfloor\big\rbrace$ such that $\big\lbrace t \big/ t', \; t' \big/ t \big\rbrace = \big\lbrace e^{2i\pi k/q}, \; e^{-2i\pi k/q} \big\rbrace$. If $q\neq 2k$, there exists a unique $t \in \C^\ast$ such that $\phi(t)=\phi\big(t e^{2i\pi k/q}\big)=\nu$. Moreover, the parameter $t$ satisfies 
\[
\begin{array}{l}
\displaystyle \prod_{j=1}^\ell \; \dfrac{t-a_j}{t \, e^\frac{2i\pi k}{q}-a_j} = e^{-\frac{2i\pi kp}{q}} \;  \Leftrightarrow	 \; \displaystyle \prod_{j=1}^\ell \big(t-a_j\big) - e^{-\frac{2i\pi kp}{q}} \, \prod_{j=1}^\ell \Big(t\, e^\frac{2i\pi k}{q}-a_j\Big) = 0 \\ \\
\Leftrightarrow \displaystyle \prod_{j=1}^\ell \Big( s \, e^{-\frac{i\pi k}{q}} -a_j \Big) - e^{-\frac{2i\pi kp}{q}} \, \prod_{j=1}^\ell \Big(s \, e^\frac{i\pi k}{q}-a_j \Big) = 0 \quad \big(\text{where } s:= t\, e^\frac{i\pi k}{q} \big) \\ \\
 \Leftrightarrow \displaystyle  \sum_{j=0}^\ell (-1)^{\ell-j} \cdot \sigma_{\ell-j}(a_1,...,a_\ell) \cdot s^j \cdot \Big(e^{-\frac{i\pi kj}{q}}- e^\frac{i\pi k(j-2p)}{q} \Big) =0 \\ \\
\Leftrightarrow \displaystyle \sum_{j=0}^\ell (-1)^{\ell-j} \sin\left(\dfrac{k\pi(p-j)}{q}\right) \cdot  \sigma_{\ell-j}(a_1,...,a_\ell) \cdot s^j =0 \; \Leftrightarrow \; P_{k, \{a\}}(t)=0,
\end{array}  \]
where the penultimate equivalence is obtained after multiplying both sides by $\big(1/2i\big)\,e^{\frac{i\pi kp}{q}}$. Conversely, for any root $t$ of $P_{k, \{a\}}$, we have that $\phi(t)=\phi\big(t e^{2i\pi k/q}\big)$. Therefore, $\phi(t)$ is a node of $C$. The case $q=2k$ is similar, except that the multiplication by $e^{2i\pi k/q}=-1$ is an involution so that we cannot distinguish $t$ and $t'$ as above. This is why we adapted the definition of $P_{k, \{a\}}$ in \eqref{eq:pol2}.

We showed above that  $\textbf{P}$ is surjective and that its restriction to  $\big\lbrace  t \in \C^\ast \; \big| \; P_{k, \{a\}}(t)=0 \big\rbrace$ is injective for each $k$. If there exists a node $\nu \in C$ parametrised by zeroes of  $P_{k, \{a\}}$ and $P_{k', \{a\}}$ for $k\neq k'$, then $\nu$ admits at least four preimages by $\phi$. This is a contradiction with the fact that $\nu$ is a node. We deduce that $\textbf{P}$ is a bijection.

Assume now that $C$ is a rational simple Harnack curve and consider a node $\nu\in C$ parametrised by a root $t'$ of the polynomial $P_{k,\{a\}}$. We need to prove that $\obs_{\Delta,1}\big(\ord(\nu)\big)=k$. To do so, recall from Definition \ref{def:ord} that $\ord(\nu)=(-b,a)$ where $(a,b) \in  H_1\big(\ttor,\Z\big)$ is the homology class of a certain loop in $\Cb$ given as the image by $\phi$ of a path $\varrho : [0,1] \rightarrow \cp{1}$ invariant by complex conjugation and joining $t'$ to $t'e^{2i\pi k/q}$.
Denoting $\phi:=(\phi_1,\phi_2)$, we deduce from Cauchy's integral formula that 
\[
\begin{array}{rl}
\displaystyle \pm a & = \quad \displaystyle \frac{1}{2i\pi}  \int_{\phi_1\circ \varrho} \frac{dz}{z} \quad  = \quad  \frac{q}{2i\pi} \int_{\varrho} \frac{d t}{t}
\\
\\
& \displaystyle = \quad \frac{q}{2i \pi}  \log\big(t'e^{2i\pi k/q} \big/ t'\big) \quad = \quad k.
\end{array}
\]
From the definition $\obs_{\Delta,1}(u,v)= \dist(v, q\Z)$, we deduce that $\obs_{\Delta,1}\big(\ord(\nu)\big)= d(a, q\Z)= d(\pm k, \alpha\Z)=k$. This concludes the proof.
\end{proof}
 
\begin{Remark} For a rational simple Harnack curve $C \in \sv{\Delta,1}$, we can actually show that all the roots of $P_{k,\{a\}}$ belong to the punctured line $e^{-i\pi kp/q} \R^\ast \subset \C^\ast $ and that their distribution on the connected components of the punctured line corresponds to the distribution of the points of $ \; \obs_{\Delta,1}^{-1}(k)\subset \itrz(\Delta)$ on the heights $k$ and $q-k$, see Figure \ref{fig:order}.
\end{Remark}

\begin{figure}[h]
\begin{center}
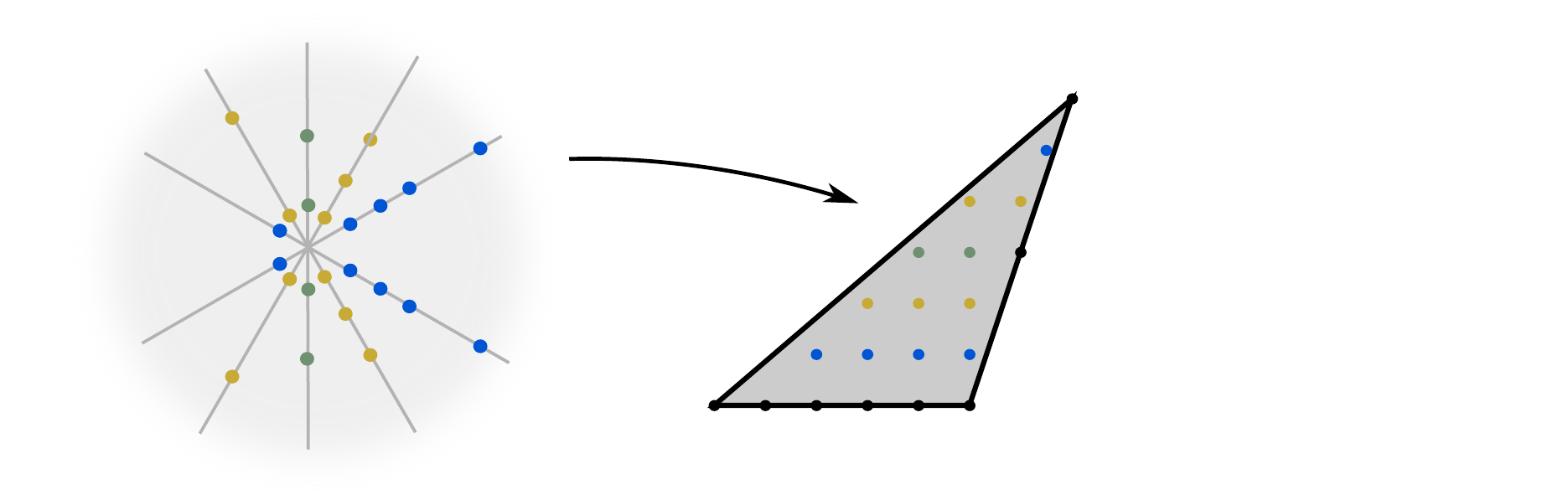
\end{center}
\caption{The correspondence  $\textbf{P}$ for $\Delta:=\conv\big(\big\lbrace (0,0), \; (5,0), \; (7,6)\big\rbrace\big)$ and  the simple Harnack curve $C \in \sv{\Delta,1}$ parametrised by $\phi_{\{1\}}(t)=\big(t^6,(t-1)^5/t^7\big)$. On the left, we represented by coloured dots the parameters $t\in \cp{1}$ mapped to nodes of $C$ by $\phi_{\{1\}}$. Each node of $C$ corresponds to a pair of complex-conjugated points in the $t$-space. The roots of $P_{k,\{1\}}$ are distributed on the punctured line $e^{-i\pi k7/6} \R^\ast$ for $k\in \big\lbrace1,2,3\big\rbrace$. In the center, we represented the lattice polygon $\Delta$ together with its interior points. Finally, we represented the set $Q_{\Delta,1}$ on the right. The colouring of the points in $\cp{1}$ and $\itrz(\Delta)$  is preserved under the maps $\ord \; \circ \; \phi_{\{1\}}$ and $\obs_{\Delta,1}$.}
\label{fig:order}
\end{figure}

Under the correspondence $\textbf{P}$ of Proposition \ref{prop:description}, Theorem \ref{thm:triangle} amounts to the fact that we can permute the roots of the polynomials $P_{k,\{a\}}$ independently on $k$ while moving in the space of parameters $\{a\}$. In order to do so, we will study the discriminants of the polynomials $P_{k,\{a\}}$ and show that we can apply \cite[Corollary 4]{EL2}.

\begin{Definition}
For any $k \in \big\lbrace 1,...,\lfloor q /2\rfloor\big\rbrace$, define the discriminant $\emph{\widetilde\mD_k} \subset \Sym_\ell\big(\C^\ast\big)$ as the set of parameters $\{a\}:=\{a_1,...,a_\ell\}$ for which the polynomial $P_{k,\{a\}}$ has a multiple root in $\C^\ast$.
\end{Definition}

Let us describe the above discriminants. Recall from \cite[Ch.4, \S 2.D, Proposition 2.7]{GKZ} that the map expressing the coefficients of a univariate polynomial in terms of its roots, namely 
\[
\begin{array}{rcl}
\mR : \Sym_\ell\big(\C^\ast\big) &\rightarrow& \C^{\ell+1}\\
\{a\}&\mapsto & \sum_{0\leq j \leq \ell}\;  (-1)^{j}\sigma_{j}(\{a\})\cdot  t^{\ell-j}
\end{array}
\]
is an isomorphism onto its image in $\C^\ast \times \C^{\ell-1}\times \{1\}$. For a fixed non-empty support $J \subset \big\lbrace 0,...,\ell\big\rbrace$, consider the projection onto the space of univariate polynomials with support contained in $J$
\[
\begin{array}{rcl}
\mP_J \; : \; \C^{\ell+1}  	& \rightarrow		& \C^J \subset \C^{\ell+1} \\
\sum_{0\leq j \leq \ell}\; c_j\cdot t^j & \mapsto	&\sum_{0\leq j \leq \ell}\; \mathds{1}_{J}(j) c_j\cdot t^j
\end{array}
\]
where $\mathds{1}_{J}$ is the indicator function of $J \subset\big\lbrace 0,...,\ell \big\rbrace$. We can therefore describe $P_{k, \{a\}}$ by 
$$P_{k, \{a\}}=\big(\lambda_k \circ \mP_{J_k} \circ  \mR\big)(\{a\})$$ 
for any $k \in \big\lbrace 0,1,...,\lfloor q /2\rfloor\big\rbrace$, where $\lambda_k$ refers abusively to the coordinates-wise multiplication by the multiplier $\lambda_k$. In turn, this allows to describe the discriminants $\widetilde \mD_k$ as follows. Denote $\mD_J$ the $J$-discriminant in the space of univariate polynomials with support $J$, see \cite[Ch.9, \S 1.A]{GKZ}. Then, we have
\begin{equation}\label{eq:disc}
\mD_k=\Big(\lambda_k \circ \mP_{J_k} \circ  \mR\Big)^{-1}\big(\mD_{J_k}\big).
\end{equation}

\begin{Proposition}\label{prop:disc}
The discriminant $\widetilde \mD_k \subset \Sym_\ell\big(\C^\ast\big)$ is empty if and only if the fiber $\obs_{\Delta,1}^{-1}(k)$ has at most one element.
The non-empty discriminants $\widetilde\mD_k$, $k \in \big\lbrace 0,1,...,\lfloor q /2\rfloor\big\rbrace$, are pairwise distinct irreducible hypersurfaces.
\end{Proposition}

In order to prove the above statement, we need to study the supports of the polynomials $P_{k, \{a\}}$ since we aim to apply \cite[Corollary 4]{EL2}.

\begin{Lemma}\label{lem:support}
For any $k \in \big\lbrace 1,...,\lfloor q /2\rfloor\big\rbrace$ such that $\obs_{\Delta,1}^{-1}(k)$ is non-empty, the support $J_k \subset \big\lbrace 0,...,\ell\big\rbrace$ generates $\Z$ as an affine lattice.
\end{Lemma}

\begin{proof}
Assume first that $q\neq 2k$. An element $j\in \big\lbrace 0,...,\ell\big\rbrace$ is in the complement of $J_k$ if and only if $\sin\big(\frac{2i\pi k(p-j)}{q}\big)=0$, if and only if $$ \textstyle pk \equiv  jk \; mod \;  q \; \Leftrightarrow \; \frac{pk}{\gcd(q,k)} \equiv  \frac{jk}{\gcd(q,k)}  \; mod \;  \frac{q}{\gcd(q,k)} \; \Leftrightarrow \; p \equiv  j \; mod \;  \frac{q}{\gcd(q,k)} .$$ 
Since $0\leq k \leq \lfloor q/2 \rfloor$ and $q\neq 2k$, the integer $\frac{q}{\gcd(q,k)}$ is at least $3$. Then, the support $J_k$ has no consecutive elements only if $\frac{q}{\gcd(q,k)}=3$, $\ell=2$ and $p \equiv 1 \; mod \; 3$ (recall that $\ell\geqslant 2$ by assumption). In particular, the support $J_k$ is the singleton $\{1\}$, the polynomial $P_{k,\{a\}}$ has no roots in $\C^\ast$ and $ \obs_{\Delta,1}^{-1}(k)$ is therefore empty, according to Proposition \ref{prop:description}.

Assume now that $q=2k$. It follows from the definition of $P_{k,\{a\}}$ that the support $J_k$ is $\{0,\cdots,\lfloor(\ell-1)/2\rfloor\}$ if $p$ is even and $\{0,\cdots,\lfloor\ell/2\rfloor\}$ if $p$ is odd. It implies that the number of roots of $P_{k,\{a\}}$, which is the cardinality of $ \obs_{\Delta,1}^{-1}(k)$, is given by the greatest integer in the support. The result follows.
\end{proof}

Before proving Proposition \ref{prop:disc}, let us gather some elementary observations on the discriminants $\widetilde \mD_k$. Recall that a general polynomial in $\mD_J$ is of the form $\sum_{j\in J} c_j \cdot c^j\cdot t^j$ where $c \in \C^\ast$ and $\sum_{j\in J} c_j \cdot t^j$ is singular at $t=1$. Since polynomials singular at $t=1$ are described by the equations
\[\sum_{j\in J} c_j  =  \sum_{j\in J} j \cdot c_j =0\]
we recover that $\mD_J$ is empty if $\vert J \vert \leq 2$ and that it is a reduced, irreducible algebraic hypersurface otherwise. It also implies that $\mD_J$ is not invariant under multiplication by a multiplier  $\lambda \in \big(\C^\ast\big)^J$ unless $\lambda=(1,...,1)$.

\begin{proof}[Proof of Proposition \ref{prop:disc}]
According to the description \eqref{eq:disc}, the discriminant $\widetilde \mD_k \subset \Sym_\ell\big(\C^\ast\big)$ is empty if and only if the discriminant $\mD_{J_k}$ is empty, if and only if $\vert J_k \vert \leq 2$. There are two cases: either $J_k=\big\lbrace j, j+1 \big\rbrace$ and then $\vert \obs_{\Delta,1}^{-1}(k)\vert =1$ by Proposition \ref{prop:description}, or $J_k$ does not generate $\Z$ in which case $\vert \obs_{\Delta,1}^{-1}(k)\vert \leq 1$ by Lemma \ref{lem:support}. 
This proves the first part of the statement. 

The fact that any non-empty discriminant is an irreducible hypersurface follows also from \eqref{eq:disc}. Indeed, the map $\lambda_k \circ \mP_{J_k} \circ  \mR$ is the composition of an isomorphism  with a linear map and the discriminants $\mD_{J_k}$ are irreducible hypersurfaces in this case.

Let us now show that if $\widetilde\mD_k$ and $\widetilde\mD_{k'}$ are non-empty and  $k\neq k'$, then $\widetilde\mD_k$ and $\widetilde\mD_{k'}$ are distinct. Again, we will distinguish two cases: either the supports $J_k$ and $J_{k'}$ coincide or they do not. Assume first that $J_k=J_{k'}$. By the description \eqref{eq:disc}, the equality $\mD_k=\mD_{k'}$ is equivalent to the $J_k$-discriminant $\mD_{J_k}$ being invariant under multiplication by $\lambda:= \lambda_k \big/ \lambda_{k'} \in (\C^\ast)^{J_k}$. By Proposition \ref{prop:description}, the roots of the polynomial $P_{k,\{a\}}$ are distinct from the root of $P_{k',\{a\}}$ for general $\{a\}$. It implies that the multipliers $\lambda_k$ and $\lambda_{k'}$ are linearly independent and then that $\lambda\neq (1,...,1)$. By the above discussion, the discriminant $\widetilde\mD_k$ is not invariant under multiplication by $\lambda$. It follows that $\widetilde\mD_k$ and $\widetilde\mD_{k'}$ are distinct.

Assume at last that $J_k\neq J_{k'}$ and take $j \in J_k \setminus J_{k'}$. For any point $ \textbf{c} \in \mR\big(\widetilde\mD_{k'}\big)$, the set $\mR\big(\widetilde\mD_{k'}\big)$ contains the line $\textbf{c} + \C^{\{j\}}$, where $\C^{\{j\}}$ is the $j^{th}$ coordinate line. However, we easily see from the above description of the $J$-discriminants  that $\mR\big(\widetilde\mD_{k}\big)$ does not contain such a line since $j\in J_{k}$. It follows that $\widetilde\mD_k$ and $\widetilde\mD_{k'}$ are distinct and the statement follows.
\end{proof}
\begin{proof}[Proof of Theorem \ref{thm:triangle}]
Let us show that we are in position to apply \cite[Corollary 4]{EL2}. First, denote by  $L=(L_1,\cdots L_{\lfloor q/2\rfloor}): \C^{J_1}\times \cdots \times \C^{J_{\lfloor q/2\rfloor}}$ the linear map with coordinate $L_k:= \lambda_k \circ \mP_{J_k}$ and denote also $\mD_k\subset \C^\ell$ the preimage under $L_k$ of the $J_k$-discriminant. In order to apply \cite[Corollary 4]{EL2}, each support $J_1, \cdots,J_{\lfloor q/2\rfloor}$ should generate $\Z$, which it does according to Lemma \ref{lem:support}. Denote at last $\mathcal{H}_k \subset \C^\ell$ the preimage under $L_k$ of the union of hyperplanes $\{c_{j_{min}}c_{j_{max}}=0\}$ where $j_{min}$ and $j_{max}$ are the smallest and largest elements in $J_k$. If we pre-compose the map $\mu_{\Delta,1}$ with the isomorphism $\mR$, the resulting map $\mu$ is defined on the fundamental group of $\C^\ell \setminus \big( \cup_k (\mD_k  \cup \mathcal{H}_k)\big)$. By Proposition \ref{prop:disc}, the map $L$ is generic in the sense of \cite[Definition 1.2]{EL2}. Therefore, \cite[Corollary 4]{EL2} implies that $\mu$ is surjective, and so is $\mu_{\Delta,1}$.
\end{proof}

\section{Patchworking monodromy}\label{sec:patch}

In this section, we show how the inclusion of a triangle $T \subset \Delta$ implies the existence of a subgroup of $\im(\mud)$ nearly isomorphic to the group $\Aut(\obs_{T,1})$ of Theorem \ref{thm:triangle}. Our strategy relies mainly on a technique similar to Viro's Patchworking, see for instance \cite{Viro08} and reference therein.

\begin{Definition}\label{def:wedge}
A \emph{wedge} in $\Delta$ is a subset of $\partial \Delta \cap M$ of the form $\big\lbrace b_0,...,b_\ell, v \big\rbrace$ such that the $b_i$ are consecutive points on some edge $\Delta_j \subset \Delta$ and $v$ is not in $\Delta_j$. 

For any subset $S\subset \partial \Delta\cap M$ such that $\left\langle S \right\rangle$ has rank $2$, define the finite set $\emph{Q_S} := ( M / \left\langle S\right\rangle) \big/ \pm \id$ and the \emph{obstruction map} $\emph{\obs_S} : \itrz(\Delta) \rightarrow Q_S$ to be the restriction of the quotient map $M \rightarrow Q_S$.

For a wedge $\w$ in $\Delta$, $T:= \conv(\w) \subset \Delta$ and any subset $S\subset M_\R$, a subgroup $G < \Aut\big(\obs_{\partial\Delta}\big)$ is a \emph{$(\w,S)$-group} (respectively a \emph{strict $(\w,S)$-group}) if  

-- every $\sigma\in G$ fixes $\itrz(\Delta)\setminus T$ (respectively $\itrz(\Delta)\setminus \itr(T)$) pointwise,

-- every $\sigma\in G$ fixes $\itrz(T)$ globally,

-- for any $\tau \in \Aut\big(\obs_S\big)$ fixing $\itrz(T)$ globally, there exists $\sigma\in G$ such that the restrictions of $\tau$ and $\sigma$ to $\itrz(T)$ coincide.\\
For short, we refer to a $(\w,\w)$-group simply as a \emph{$\w$-group}.
\end{Definition}

In the above definition, $\left\langle S\right\rangle$ is the affine lattice generated by $S$. In order to define $Q_S$, we consider the group structure on $M/\left\langle S\right\rangle$ induced from $M$, with neutral element $\left\langle S\right\rangle$. Note at last that the lattice polygon $T\subset \Delta$ is a triangle and that 
 $\obs_{\w}= \obs_{\Delta,1}$ (for the appropriate indexation) and $G_\w= \Aut\big(\obs_{\w}\big)$ when $\Delta=\conv(\w)$.

The main result of this section is the following.

\begin{Theorem}\label{thm:patch}
For any wedge $\w$ in $\Delta$, the image of $\mud$ contains a $\w$-group.
\end{Theorem}

There are three cases to consider depending whether $T:=\conv(\w)$ has exactly $1$, $2$ or $3$ of its edges on $\partial \Delta$. In the latter case, the polygon $\Delta$ is itself a triangle and Theorem \ref{thm:patch} is a consequence of Theorem \ref{thm:triangle}. In the rest of this section, we restrict to the first case. The second case requires no extra arguments, simply different notation.

Let us assume that $T$ has exactly $1$ edge $\edge$ on $\partial \Delta$. Denote $\edge' $ and $\edge''$ the remaining edges of $T$ so that $\edge'$, $\edge$ and $\edge''$ are ordered counter-clockwise on $\partial T$. We fix coordinates on $M$ such that the edge $\edge$ is the segment joining $(0,0)$ to $(\ell,0)$ and such that the vertex $\edge' \cap \edge''$ has coordinates $(p,q)$ for some $p, \, q \in \Z_{>0}$. The polygon $T$ induces the subdivision $\Delta := \Delta' \cup T \cup \Delta''$ into lattice polygons satisfying $\Delta' \cap T= \edge'$ and $\Delta'' \cap T= \edge''$. The latter subdivision is given by the domains of linearity of the piecewise linear convex function $\nu:  \Delta \rightarrow \R$ defined by 
\[
\nu (a,b) = \left\lbrace 
\begin{array}{ll}
0 & \text{if } (a,b) \in T \\
pb-qa & \text{if } (a,b) \in \Delta' \\
q(a-\ell)+(\ell-p)b & \text{if } (a,b) \in \Delta'' \\
\end{array}
\right. .
\]
For the positive integer $m:= 1+ \max \left\lbrace \nu (a,b) \; \big| \; (a,b)\in \Delta \right\rbrace$, define in turn the lattice polytope 
\[ \Delta_\nu := \left\lbrace (a,b,c) \in \Delta \times \R \; \big| \; (a,b) \in \Delta, \; \nu(a,b) \leq c \leq m \right\rbrace.\]
By construction, the projection onto $\Delta$ identifies the non-vertical facets of $\Delta_\nu$ with $\Delta'$, $T$, $\Delta''$ and $\Delta$. Therefore, the corresponding toric divisors of the toric $3$-fold $X_{\Delta_\nu}$ identify with $X_{\Delta'}$, $X_T$, $X_{\Delta''}$ and $X_\Delta$ respectively. Denote $X_{\edge'}:= X_{\Delta'} \cap X_{T}$ and $X_{\edge''}:= X_{\Delta''} \cap X_{T}$ the torus-orbit of dimension $1$ in $X_{\Delta_\nu}$.

The coordinates on $\Delta\times \R \supset \Delta_\nu$ induce coordinates $(x,y,z)\in (\C^\ast)^3$ on the torus of $X_{\Delta_\nu}$. It follows from \cite[\S 3.3]{Ful} that the $z$-coordinate realises a linear equivalence in $\Delta_\nu$ between the divisors $X_\Delta$ (at $z=\infty$) and $X_{\Delta'}\cup X_T \cup X_{\Delta''}$ (at $z=0$). In particular, the closure of the one-parameter subgroup $\big\lbrace x=a, \, y=b \big\rbrace \subset X_{\Delta_\nu}$ intersects both $X_\Delta$ and $X_T$ transversally at one point, for any $(a,b) \in \ttor$. Therefore, we can define a vertical projection $\pi$ along the latter subgroups,  both upward and downward, landing in the respective tori of $X_\Delta$ and $X_T$. We denote both projections by $\pi$ and equip the tori of $X_\Delta$ and $X_T$ with coordinates $(x,y)$ such that $\pi(x,y,z)=(x,y)$ both upward and downward. Note that $\pi$ induces an isomorphism from any horizontal slice $\overline{\{z=c\}} \subset X_{\Delta_\nu}$ to $X_\Delta$ and an other isomorphism from $\big\lbrace (x,y,c) \in (\C^\ast)^3 \big\rbrace$ to $\ttor \subset X_T$.

Similarly, we have that any interior point of the divisor $X_{\Delta'}$ (respectively $X_{\Delta''}$) is the limit point of a subgroup of the form $\big\lbrace x=az^q, \; y=bz^{-p}\big\rbrace$ $\big($respectively $\big\lbrace x=az^{-q}, \; y=bz^{p-\ell}\big\rbrace\big)$. As above, we equip the tori of $X_{\Delta'}$ and $X_{\Delta''}$ with coordinates $(x,y)$ such that the map $\pi' : (\C^\ast)^3 \rightarrow \ttor \subset X_{\Delta'}$ given by $\pi' (x,y,z)=(xz^{-q},yz^{p})$ is the projection along the subgroups $\big\lbrace x=az^q, \; y=bz^{-p}\big\rbrace$ and the map $\pi'' : (\C^\ast)^3 \rightarrow \ttor \subset X_{\Delta''}$ given by $\pi''(x,y,z)=(xz^{q},yz^{\ell-p})$ is the projection along the subgroups $\big\lbrace x=az^{-q}, \; y=bz^{p-\ell}\big\rbrace\big)$. Again the maps $\pi'$ and $\pi''$ induce isomorphisms from $\big\lbrace (x,y,c) \in (\C^\ast)^3 \big\rbrace$ to $\ttor \subset X_{\Delta'}$ and $\ttor \subset X_{\Delta''}$ respectively.

We now give a counterpart to Viro's patchworking polynomials in terms of parametrisation of rational  curves. In the formula \eqref{eq:param4} below, we relabel the parameters $a_{j,l}$ and the corresponding exponents $(\alpha_j, \beta_j)$ of the parametrisation \eqref{eq:param}. Denote by $J'$, $J_T$ and $J''$ the multi-sets of primitive integer vectors in $\partial \Delta$ contained in $\partial \Delta'$, $\partial T$ and $\partial \Delta''$ respectively. For any $j \in J' \cup J_T \cup J''$, the vector $(\alpha_j, \beta_j)$ is the primitive inner normal to the edge of $\Delta$ containing $j$. Note in particular that $(\alpha_j, \beta_j)=(0,1)$ for any $j \in J_T$. At last, we denote by $\sv{T,\w} \subset \sv{T}$ the subset of curves intersecting the toric orbits corresponding to $\edge'$ and $\edge''$ only once. This is the analogue of $\sv{\Delta,1}$ of Section \ref{sec:maintool}.

\begin{Lemma}\label{lem:patch}
Let $\{a\}:=\big\lbrace a_j\big\rbrace_{j \in J' \cup J_T \cup J''} \subset \C^\ast$ such that $\big\lbrace a_j\big\rbrace_{j \in J'}$, $\big\lbrace a_j\big\rbrace_{j \in J_T}$ and $\big\lbrace a_j\big\rbrace_{j \in J''}$ are mutually disjoint. For any $z \in \C^\ast$, define the parametrisation $\phi_{z,\{a\}}$ from $\cp{1}$ to $X_{\Delta_\nu}$ by 
\begin{equation}\label{eq:param4}
t \mapsto \Big( \prod_{j\in J'} \big(t-za_j\big)^{\alpha_j}\prod_{j\in J''} \big(1-zta_j^{-1}\big)^{\alpha_j}, \; \prod_{j\in J'} \big(t-za_j\big)^{\beta_j}\prod_{j\in J_T} \big(t-a_j\big)\prod_{j\in J''} \big(1-zta_j^{-1}\big)^{\beta_j}, \; z \; \Big)
\end{equation}
and define $C_z := \im\big( \phi_{z,\{a\}}\big)$. Then, the rational curve $C_z$ converges to the curve $C_0 \subset X_{\Delta'}\cup X_T \cup X_{\Delta''}$ with irreducible components $C' \subset X_{\Delta'}$, $C_T \subset X_T$, $C'' \subset X_{\Delta''}$ parametrised respectively by
\[
\begin{array}{l}
\displaystyle 
\phi_{\{a\}}'(t):= \Big( \prod_{j\in J'} \big(t-a_j\big)^{\alpha_j}, \;   \prod_{j\in J_T} \big(-a_j\big) \prod_{j\in J'} \big(t-a_j\big)^{\beta_j}\Big), \\  \\
\displaystyle 
\phi^T_{\{a\}}(t):= \Big( t^q, t^{-p} \; \prod_{j\in J_T} \big(t-a_j\big)\Big) , \\  \\
\displaystyle 
\phi_{\{a\}}''(t):= \Big( t^{q} \prod_{j\in J''} \big(1-ta_j^{-1}\big)^{\alpha_j}, t^{\ell-p} \; \prod_{j\in J''} \big(1-ta_j^{-1}\big)^{\beta_j}\Big).
\end{array}
\]
In particular, the curve $C_0$ intersects the divisor $X_{\edge'}$  $\big($respectively $X_{\edge''}\big)$ at the single point $p':=C' \cap C_T$ (respectively $p'':=C'' \cap C_T$).  For generic parameters $a_j$, the irreducible components of $C_0$ are nodal curves and the curve $C_T$ is an element of the subspace $\sv{T,\w}\subset\sv{T}$.
\end{Lemma}

\begin{proof}
As $z$ tends to $0$, the rational curve $C_z$ converges towards the divisor $X_{\Delta'}\cup X_T \cup X_{\Delta''}$. Hence, the limiting curve $C_0$ consists of rational components $C' \subset X_{\Delta'}$, $C_T \subset X_T$ and $C'' \subset X_{\Delta''}$. Let us compute respective parametrisations $\phi_{\{a\}}'$, $\phi_{\{a\}}^T$ and $\phi_{\{a\}}''$.

The curves $C'$, $C_T $ and $C''$ are the respective Hausdorff limit of $\pi' \circ \phi_{z,\{a\}}\big(\cp{1}\big)$, $\pi\circ\phi_{z,\{a\}}\big(\cp{1}\big)$ and $\pi''\circ\phi_{z,\{a\}}\big(\cp{1}\big)$ when $z$ tends to $0$.
For $\phi_{\{a\}}^T$, we have that $\lim_{z\rightarrow 0} \pi\big(\phi_{z,\{a\}}(t)\big)=\big( t^q, t^{-p} \; \prod_{j\in J_T} \big(t-a_j\big)\big)$. It follows that $\phi_{\{a\}}^T$ is as announced in the statement. For the map $\phi_{\{a\}}'$, we need to make the change of variable $t=zt'$ within the  limit $\lim_{z \rightarrow 0} \pi'\big(\phi_{z,\{a\}}(t)\big)$. The computation goes as follows
\[
\begin{array}{l}
\displaystyle \quad \lim_{z \rightarrow 0} \Big(z^{-q} \prod_{j\in J'} \big(t-za_j\big)^{\alpha_j}\prod_{j\in J''} \big(1-zta_j^{-1}\big)^{\alpha_j}, \; z^{p} \prod_{j\in J'} \big(t-za_j\big)^{\beta_j}\prod_{j\in J_T} \big(t-a_j\big)\prod_{j\in J''} \big(1-zta_j^{-1}\big)^{\beta_j} \Big) \\ \\
= \displaystyle \lim_{z \rightarrow 0} \Big(z^{-q} \prod_{J'} \big(zt'-za_j\big)^{\alpha_j}\prod_{ J''} \big(1-z^2t'a_j^{-1}\big)^{\alpha_j}, \; z^{p} \prod_{J'} \big(zt'-za_j\big)^{\beta_j}\prod_{J_T} \big(zt'-a_j\big)\prod_{J''} \big(1-z^2t'a_j^{-1}\big)^{\beta_j} \Big) \\ \\
= \displaystyle \lim_{z \rightarrow 0} \Big(\prod_{J'} \big(t'-a_j\big)^{\alpha_j}\prod_{J''} \big(1-z^2t'a_j^{-1}\big)^{\alpha_j}, \; \prod_{J'} \big(t'-a_j\big)^{\beta_j}\prod_{J_T} \big(zt'-a_j\big)\prod_{J''} \big(1-z^2t'a_j^{-1}\big)^{\beta_j} \Big) \\ \\
= \displaystyle \Big( \prod_{j\in J'} \big(t'-a_j\big)^{\alpha_j}, \;   \prod_{j\in J_T} \big(-a_j\big) \prod_{j\in J'} \big(t'-a_j\big)^{\beta_j}\Big).
\end{array}
\]
It follows that $\phi'_{\{a\}}$ is as announced above. We obtain the parametrisation $\phi_{\{a\}}''$ similarly, using the change of variable $t''=zt$.

For the second part of the statement, recall that we have $a_j\in\C^\ast$ for any $j$. Under this assumption, we read from the parametrisation $\phi_{\{a\}}'$ that $\infty \in \cp{1}$ is the only point mapping to $X_\edge'$, from $\phi_{\{a\}}''$ that $0 \in \cp{1}$ is the only point mapping to $X_\edge''$ and finally that $\phi_{\{a\}}^T$ is as in \eqref{eq:param3}. 
In particular, the curve $C_T$ intersects both $X_{\edge'}$ and $X_{\edge''}$ at a single point. Denote $p':=C' \cap X_{\edge'}$ and $p'':=C'' \cap X_{\edge''}$. As the curve $C_0$ is connected, we have $p'=C' \cap C_T$ and $p''=C'' \cap C_T$. Finally, it is also clear from the parametrisations $\phi_{\{a\}}'$, $\phi_{\{a\}}^T$ and $\phi_{\{a\}}''$ and from the general form \eqref{eq:param} that generic parameters $a_j$ lead to generic rational curves $C'$, $C_T$ and $C''$ submitted to the above tangency conditions with $X_{\edge'}$ and $X_{\edge''}$. 
In particular, the curves $C'$, $C_T$ and $C''$ have only nodes as singularities.
\end{proof}

Let $\emph{U}$ be the space of parameters $\big(z, \{a_j\}_{1 \leq j \leq n}\big)$ involved in \eqref{eq:param4}.
 For $\big(z,\{a\} \big) \in U$, define 
$$\emph{\Phi\big(z,\{a\} \big)} := \phi_{z,\{a\}}\big(\cp{1}\big) \subset X_{\Delta_\nu}$$ 
to be the curve parametrised as in \eqref{eq:param4}. By extension, define 
$$\emph{\Phi \big(0,\{a\} \big)} := \phi_{\{a\}}'\big(\cp{1}\big) \cup  \phi_{\{a\}}^T\big(\cp{1}\big) \cup  \phi_{\{a\}}''\big(\cp{1}\big) \subset X_{\Delta'} \cup  X_T \cup  X_{\Delta''}$$
as in Lemma \ref{lem:patch}.
When the parameters $\{a\}$ are real and cyclically ordered as in Proposition \ref{prop:ratharnack},  the curve $\pi\big(\Phi(z,\{a\}) \big)\subset X_\Delta$ is a rational simple Harnack curve for all $0<z\leq 1$. It follows from the same proposition that the three irreducible components of $\Phi \big(0,\{a\} \big)$ are simple Harnack curves in their respective ambient toric surfaces $X_{\Delta'}$, $X_T$ and $X_{\Delta''}$. Then, the rational  curve $\Phi \big(0,\{a\} \big):= C'\cup C_T \cup C''$ admits an order map 
\[ \emph{\ord_0} \, : \, \big\lbrace \text{nodes of } C', \; C_T \text{ and } C'' \big\rbrace \rightarrow \itrz(\Delta') \cup  \itrz(T) \cup \itrz(\Delta'') \subset \itrz(\Delta)\]
defined by the order maps of Definition \ref{def:ord} on each irreducible component $C'$, $C_T$ and $C''$. The map $\ord_0$ is a bijection by Proposition \ref{prop:extord}.

\begin{Lemma}\label{lem:ord}
Let $\{a\}$ be such that $\pi\big(\Phi(1,\{a\}) \big)\subset X_\Delta$ is a simple Harnack curve. For any $\alpha \in \itrz(\Delta)$ and $0<z\leq 1$, denote by $p_{z,\alpha}$  the unique point of $\Phi\big(z,\{a\}\big)\subset X_{\Delta_\nu}$ such that $\ord\big( \pi(p_{z,\alpha})\big)=\alpha$. For any $\alpha \in \itrz(\Delta) \setminus (\edge' \cup \edge'')$, denote by $p_{0,\alpha}$ the unique point of $\Phi\big(0,\{a\}\big)$ such that $\ord_0 \big( \pi(p_{0,\alpha})\big)=\alpha$. Then, we have
\begin{enumerate}
\item[$a)$] If $\alpha \in  \itrz(\Delta) \setminus (\edge' \cup \edge'')$, then $\lim_{z\rightarrow 0} p_{z,\alpha}=p_{0,\alpha}$.
\item[$b)$] If $\alpha \in \edge'$ (respectively $\edge''$), then $\lim_{z\rightarrow 0} p_{z,\alpha} \in X_{\edge'}$ (respectively $X_{\edge''}$).
\end{enumerate}
\end{Lemma}

Before tackling the proof, recall that the maps $\pi$, $\pi'$ and $\pi''$ induce isomorphisms between the tori of $X_{\Delta}$, $X_{\Delta'}$ $X_{T}$ and $X_{\Delta''}$ and that the induced isomorphisms between the respective first homology groups read as the identity in the coordinate systems chosen above.

\begin{proof}
Fix $\alpha \in  \itrz(\Delta) \setminus (\edge' \cup \edge'')$.
Recall from Definition \ref{def:ord} that $\ord_0 (p_{0,\alpha})$ is given by the homology class of a loop $\rho_0$ contained in one of the irreducible components of $\Phi\big(0,\{a\}\big)$, that this loop passes through $p_{0,\alpha}$ and is invariant by complex conjugation. We can continuously deform this loop to a loop $\rho_z \subset \Phi\big(z,\{a\}\big)$ ($0<z<1$) invariant by complex conjugation passing through the double point $p_{z,\beta}$ for some $\beta \in \itrz(\Delta)$. According to our choices of coordinate systems, we have that $\rho_0$ and $\pi(\rho_z)$ have the same homology class. It follows that $\ord\big(\pi(p_{z,\beta})\big)= \alpha$ and that $\beta=\alpha$. The statement $a)$ is proven. 

From part $a)$, we know that the point $p_{z,\alpha}$ for $\alpha \in \edge' \cup \edge''$ has to converge to one of the remaining singular points of $\Phi\big(0,\{a\}\big)$, namely $p'$ and $p''$(see Lemma \ref{lem:patch}). Again, looking at homology classes of appropriate loops passing through $p_{z,\alpha}$, we deduce that $p_{z,\alpha}$ converges to $p'$ if and only if $\alpha \in \edge'$.
\end{proof}

\begin{proof}[Proof of Theorem \ref{thm:patch}]
Let $\w$ be a wedge in $\Delta$ and denote $T:=\conv(\w)$.
Recall that the monodromy map $\mud : \pi_1 \big( \sv{\Delta}, \rc \big) \rightarrow \Aut \big(\itrz(\Delta)\big)$ of \eqref{eq:monomapdelta} is defined for a simple Harnack curve $\rc \in \sv{\Delta}$ via the bijection $\ord : \big\lbrace \text{nodes of } \rc\big\rbrace \rightarrow \itrz(\Delta)$. Let $\{a\}$ be real parameters so that $\rc = \pi\big(\Phi(1, \{a\})\big)$. Let us fix an arbitrary element $\tau \in \Aut\big((\obs_\w)_{\vert \itrz(T)}\big)$ and show that there exists $\sigma \in \im(\mud)$ satisfying the properties $a)$, $b)$ of Definition \ref{def:wedge} and such that $\sigma_{\vert_{\itrz(T)}}=\tau$. As $\tau$ is arbitrary, the latter implies the statement we aim to prove.

According to Theorem \ref{thm:triangle}, there exists a loop $\big\lbrace a(\theta)\big\rbrace:=\big\lbrace a_j(\theta)\big\rbrace_{j \in J' \cup J_T\cup J''}$ with $0\leq \theta \leq 1$ based at $\{a\}$, that is $\big\lbrace a(1)\big\rbrace=\big\lbrace a(0)\big\rbrace=\{a\}$,  and satisfying the following:
\begin{enumerate}
\item[-] the family $\big\lbrace a_j(\theta)\big\rbrace_{j \in J'\cup J''}$ is constant,
\item[-] for any $0\leq \theta \leq 1$, the rational curve $\phi_{\{ a(\theta)\}}^T\big(\cp{1}\big)\subset X_T$ is in $\sv{T, \w}$,
\item[-] the image by monodromy map $\mu_T$ of the family $\big\lbrace\phi_{\{a(\theta)\}}^T\big(\cp{1}\big) \big\rbrace_{0\leq \theta \leq 1} \subset  \sv{T, \w}$ based at the simple Harnack curve $\phi_{\{a\}}^T\big(\cp{1}\big)$ is the permutation $\tau$.
\end{enumerate}
As $\big\lbrace a_j(\theta)\big\rbrace_{j \in J'\cup J''}$ is constant, the monodromy is trivial on the families $\big\lbrace\phi_{\{a(\theta)\}}'\big(\cp{1}\big) \big\rbrace_{0\leq \theta \leq 1} \subset  \sv{\Delta'}$ and $\big\lbrace\phi_{\{a(\theta)\}}''\big(\cp{1}\big) \big\rbrace_{0\leq \theta \leq 1} \subset  \sv{\Delta''}$. 
Indeed, the parametrisation $\phi_{\{ a(\theta)\}}''$ only depends on the parameters $\big\lbrace a_j(\theta)\big\rbrace_{j \in J''}$ while the multiplicative factor $\prod_{j\in J_T} \big(-a_j\big)$ in the parametrisation $\phi_{\{ a(\theta)\}}'$ induces no permutation of the nodes of the curve $\phi_{\{ a(\theta)\}}'\big(\cp{1}\big)$.
Consider now the loop $\big\lbrace\big(z(\theta), \{b(\theta)\}\big)\big\rbrace_{0\leq \theta \leq 3} \subset U$ based at $\big(1, \{a\}\big)$ given by 
\[\big(z(\theta), \{b(\theta)\}\big):=\left\lbrace
\begin{array}{ll}
\big(1-\theta, \{a\})\big) & \text{for } \; 0\leq \theta \leq 1 \\
\big(0, \{a(\theta-1)\})\big) & \text{for } \; 1\leq \theta \leq 2 \\
\big(\theta-2, \{a\}\big) & \text{for } \; 2\leq \theta \leq 3 \\
\end{array}.
\right.\]
There exists an arbitrarily small deformation $\big\lbrace\big(z'(\theta), \{b'(\theta)\}\big)\big\rbrace$ of $\big\lbrace\big(z(\theta), \{b(\theta)\}\big)\big\rbrace$ inside $U \setminus \big\lbrace z=0 \big\rbrace$ such that the curve $C(\theta):= \pi \big(\Phi \big(z(\theta), \{b(\theta)\}\big)\big) \subset X_\Delta$ is in $\sv{\Delta}$ for any $0\leq \theta \leq 3$. For such deformation, denote $\sigma:= \mud \big( \big\lbrace C(\theta)\big\rbrace_{0\leq \theta \leq 3} \big)$. First, the permutation $\sigma$ is in $\Aut\big(\obs_{\partial\Delta}\big)$ by Proposition \ref{prop:obs}.  It follows now from Lemma \ref{lem:ord} that $\sigma$ satisfies $a)$ and $b)$ of Definition \ref{def:wedge} and that $\sigma_{\vert_{\itrz(T)}}=\tau$.
\end{proof}

\begin{Remark}
Fix $k \in Q_{\w} \setminus 0$ and consider a parameter $\{a\}$ such that $\phi_{\{a\}}'\big(\cp{1}\big)$ and $\phi_{\{a\}}''\big(\cp{1}\big)$ are nodal and such that the subset of parameters $\{a_j\}_{j\in J_T}$ is a generic point in  $ \D_k$.  Consider moreover a $1$-parametric family $\big\lbrace a(\theta) \big\rbrace_{0\leq \theta \leq 1}$ with $\big\lbrace a(0) \big\rbrace=\{a\}$ such that the family of rational curves $\Phi\big(\theta, \big\lbrace a(\theta) \big\rbrace\big)$ has exactly $\vert \itrz(\Delta) \vert-1$ singular points for $0<\theta\leq 1$. Equivalently, the deformation of $\Phi\big(0, \{a\}\big)$ along $\Phi\big(\theta, \big\lbrace a(\theta) \big\rbrace\big)$ preserves the $A_3$-singularity coming from $\D_k$ and deforms the points $p'$ and $p''$ (see Lemma \ref{lem:patch}) into the maximal number of double points. If such a deformation exists for any $k$ as above, then we can ensure that the permutation $\sigma$ constructed in the proof of Theorem \ref{thm:patch} is the identity on $\partial T \cap M$. The existence of such degenerations is addressed in \cite[Theorem 2.8]{ST2} and requires the vanishing of certain cohomology groups. In a rather circuitous manner, the existence of the latter deformations arises as a consequence of Theorem \ref{thm:main}. 
\end{Remark}

\section{Combinatorics}\label{sec:combinatorics}

Recall Definition \ref{def:wedge} and the notation therein. The main result of this section is the following.

\begin{Theorem}\label{thm:combinatorics}
For any complete toric surface $X$, there exists a constant $\ell := \ell(X)>0$ such that for any lattice polygon $\Delta\in \aC_{\geqslant \ell}(X)$, any subgroup $G < \Aut(\obs_{\partial\Delta})$ that contains a $\w$-group for any wedge $\w$ in $\Delta$ is the whole group $\Aut\big(\obs_{\partial\Delta}\big)$.
\end{Theorem}

From now on, the group $G < \Aut(\obs_{\partial\Delta})$ is assumed to contain a $\w$-group for any wedge $\w$ in $\Delta$.
In order to prove the above theorem, we will use the following notation\smallskip

$ \emph{\w(j,v)}:=\{\Delta_j\cap M, v\}, \;  \emph{w_j}:=w(j, \Delta_{j-2,j-1}), \;  \emph{T(j,v)}:=\conv\big(w(j,v)\big)$ and $ \emph{T_j}:= \conv\big(w_j\big)$\smallskip\\ 
%
for all $j\in \Z/n\Z$, $v\in M\cap \partial \Delta$, and successively use some of the following assumptions:\smallskip

\textbf{(A)}: for any $j\in \Z/n\Z$, the restrictions of $\obs_{\w_j}$ and $\obs_{\w_{j+1}}$ to $\itrz(T_j\cap T_{j+1})$ are at least $2$-to-$1$. In particular, they are surjective.\smallskip

\textbf{(B)}: if $q_j$ is the index of the affine lattice $\left\langle (\Delta_j \cup \Delta_{j+1}) \cap M\right\rangle$ in $M$ and $q:= \min_{j} \, q_j$ over all $j \in \Z/n\Z$, then $\Delta \in \nc{4}(X)$  if $q=1$ and $\Delta \in \nc{3q-2}(X)$ $q>1$.\smallskip

\textbf{(C)}: for any wedges $w:=w(j,v)$ and $w':=w(j,v')$ such that $v$ and $v'$ are consecutive on $\partial \Delta$, the restrictions of $\obs_{\w}$ and $\obs_{\w'}$ to $\itrz\big(T(j,v)\cap T(j,v')\big) $ are surjective.\smallskip

At last, for a finite set $E$ and a subset $F$, define the \emph{restriction} of a group $\widetilde G< \Aut(E)$ to $F$ as the subgroup of $\Aut(F)$ given by the restriction to $F$ of all elements in $\widetilde G$ that fix $F$ globally. While talking about the restriction of an element $g\in \widetilde G$ to $F$, we will implicitly imply that $g$ fixes $F$ globally. Observe for instance that a strict $\w_j$-group is isomorphic to its restriction to $\itrz(T_j)$.

The proof of Theorem \ref{thm:combinatorics} will be given as a concatenation of intermediate results that we state below. For the sake of clarity, the proof of these results is postponed until we prove Theorem \ref{thm:combinatorics}.

\begin{Proposition}\label{prop:1}
If \textbf{(A)} holds, then the restrictions of $G$ and $\Aut(\obs_{\partial\Delta})$ to $\itrz(T_j)$ coincide for any $j\in \Z/n\Z$.
\end{Proposition}

\begin{Proposition}\label{prop:2}
If \textbf{(A)} and \textbf{(B)} hold, then $G$ contains a strict $(w_j,\partial \Delta)$-group for any $j\in \nolinebreak\Z/n\Z$.
\end{Proposition}

\begin{Proposition}\label{prop:3}
If \textbf{(A)}, \textbf{(B)} and \textbf{(C)} hold, then $G$ contains $\Aut(\obs_{\partial\Delta})$.
\end{Proposition}

\begin{Proposition}\label{prop:l}
The assumptions \textbf{(A)}, \textbf{(B)} and \textbf{(C)} are satisfied if $\Delta \in \aC_{\geqslant \ell}(X)$ for   
\[\ell: = 5 \cdot \max_{ j \in \Z/n\Z} \left\lbrace \vert v_j \vert^2  \right\rbrace\]
where the norms $\vert v_j \vert$ are computed with respect to any choice of coordinates $M\simeq \Z^2$.
\end{Proposition}

\begin{Remark}
The constant $\ell$ of Proposition \ref{prop:l} is not intrinsic to the toric surface $X$ as the norms $\vert v_j \vert$ depend on the choice of coordinates on $M$. However, the constant $\ell(X)$ given as the minimum of $\lceil\ell\rceil$ for all possible choice of coordinates $M\simeq \Z^2$ is intrinsic to $X$ but harder to compute.
\end{Remark}

\begin{proof}[Proof of Theorem \ref{thm:combinatorics}]
Taking $\ell:=\ell(X)$ as defined in the above remark, the statement follows from Propositions \ref{prop:3} and \ref{prop:l}.
\end{proof}

The proof of Proposition \ref{prop:1} will require some extra material. Recall that the \emph{pushout} of a pair of maps $f:E\rightarrow F$ and $g:E\rightarrow G$ between finite sets is a pair of maps $\iota_f: F\rightarrow H$ and $\iota_g: G\rightarrow H$ satisfying $\iota_f\circ f = \iota_g\circ g$ that is universal for this property, that is for any pair of maps $j_f: F\rightarrow H'$ and $j_g: G\rightarrow H'$ such that $j_f\circ f = j_g\circ g$, there exists a unique map $u: H\rightarrow H'$ such that $j_f=u\circ \iota_f$ and $j_g=u\circ \iota_g$. In particular, the pushout is unique up to a unique isomorphism. We denote by $\emph{f\star g}$ the map $\iota_f\circ f = \iota_g\circ g$ and refer to it  as the \emph{pushout map}.

The pushout can be constructed explicitly. Indeed, we can take $H=\big(F \coprod G\big) \big/ \sim$ where $\sim$ is the finest equivalence relation that identifies pairs $(x,y)\in F\times G$ with a common preimage in $E$, and $\iota_f: F\rightarrow H$ and $\iota_g: F\rightarrow G$ are the natural projections. 

Observe that the pushout map $f\star g : E \rightarrow H$ is surjective if and only if $f$ and $g$ are. In the latter case, the map $f\star g$ can also be constructed as the quotient map $E \rightarrow E\big/\approx$ where $\approx$ is the finest equivalence relation that identifies pairs $(x,y)\in E^2$ such that $y\in g^{-1}(g(f^{-1}(f(x))))$.

\begin{Lemma}\label{lem:deckpushout}
Let $f:E\rightarrow F$ and $g:E\rightarrow G$ be two surjective maps between finite sets. Then, we have $\left\langle \Aut(f), \Aut(g)\right\rangle= \Aut(f\star g)$.
\end{Lemma}

\begin{proof}
First, it is clear that $\left\langle \Aut(f), \Aut(g)\right\rangle$ is a subgroup of  $\Aut(f\star g)$ since $f\star g$ factors through $f$ and $g$. From the above description $E \rightarrow E\big/\approx$ of the pushout map $f\star g$, we deduce that it is enough to prove that $\left\langle \Aut(f), \Aut(g)\right\rangle$ contains the transposition sending $x$ to $y$ for any pair of distinct elements $x,y \in E$ such that $y\in g^{-1}(g(f^{-1}(f(x))))$. By assumption, there exists $x' \in (f^{-1}(f(x))$ such that $y\in g^{-1}(g(x'))$. We assume that $x$, $x'$ and $y$ are pairwise distinct (the two other cases will be obvious). Then, the sought transposition $(x,y)$ sending $x$ to $y$ is given as the element $(x',y)\circ (x,x')\circ (x',y)$ where $(x,x')\in \Aut(f)$ and $(x',y)\in \Aut(g)$. The transposition $(x,y)$ is therefore an element of $\left\langle \Aut(f), \Aut(g)\right\rangle$.
\end{proof}

\begin{Lemma}\label{lem:push}
For any two subsets $S, S' \subset M_\R$ such that $\obs_{S}$ and $\obs_{S'}$ are surjective and $S\cap S'\neq \varnothing$, we have that $\obs_{S} \star \obs_{S'} = \obs_{S\cup S'}$.
\end{Lemma}

\begin{proof}
There is no loss of generality in assuming that $0\in S\cap S'$. In particular, $\left\langle S\right\rangle$ and $\left\langle S'\right\rangle$ are honest lattices and  the intermediate maps $f: \itrz(\Delta) \rightarrow M\big/ \left\langle S\right\rangle$ and $g: \itrz(\Delta) \rightarrow M\big/ \left\langle S'\right\rangle$ are group homomorphisms. Then, it is clear that $f\star g$ is the quotient  $\itrz(\Delta) \rightarrow M\big/ \left\langle S\cup S'\right\rangle$ since the latter is a reformulation of the quotient map $\itrz(\Delta) \rightarrow \itrz(\Delta)\big/ \approx$ for the appropriate equivalence relation $\approx$. Composing every map above with the quotient by $\pm \id$, we deduce the assertion.
\end{proof}

\begin{proof}[Proof of Proposition \ref{prop:1}]
We have to show that for any pair of distinct elements $x,y \in \itrz(T_j)$ in the same fiber of $\obs_{\partial\Delta}$, the restriction of $G$ to $\itrz(T_j)$ contains the transposition $(x,y)$. First, we claim that the restrictions of $G$ and $\Aut(\obs_{\partial\Delta})$ to $\itrz(T_j\cap T_{j+1})$ coincide for any $j\in \Z/n\Z$. The statement follows from the claim if $x,y \in \itrz(T_j\cap T_{j+1})$. Let us therefore assume that  
$x, y \in \itrz(T_j\setminus T_{j+1})$ (the remaining cases will be obvious). By \textbf{(A)}, we can find distinct $x', y'\in \itrz(T_j\cap T_{j+1})$ and elements $\sigma_x, \sigma_y$ in the $\w_j$-subgroup of $G$ whose restrictions to $\itrz(T_j)$ are the transpositions  $(x,x')$ and $(y,y')$ respectively. This is possible since the restriction of $\obs_{\w_j}$ to $\itrz(T_j)$ is at least $2$-to-$1$. Since $x'$ and $y'$ are still in the same fiber of $\obs_{\partial\Delta}$, it follows from the claim that there exists an element $\tau\in G$ whose restriction to  $\itrz(T_j\cap T_{j+1})$ is the transposition $(x',y')$. In turn, the permutation $\tau\circ \sigma_x \circ \tau^{-1}$ fixes $\itrz(T_j)$, even though $\tau$ may not,  and restricts to $(x,y')$. We conclude that $\sigma_y\circ (\tau\circ \sigma_x \circ \tau^{-1}) \circ \sigma_y^{-1}$  fixes $\itrz(T_j)$ and restricts to the sought transposition $(x,y)$.

It remains to prove the claim. By \textbf{(A)}, the restrictions of $\obs_{\w_1}$ and $\obs_{\w_2}$ to $\itrz(T_1 \cap T_2)$ are surjective. Since $G$ contains a $w_1$- and a $w_2$-group, it follows from Lemmas \ref{lem:deckpushout} and \ref{lem:push} that the restriction of $G$ to $\itrz(T_1 \cap T_2)$ contains the restriction of $\Aut(\obs_{\w_1\cup\w_2})$ to $\itrz(T_1 \cap T_2)$. By similar arguments as in the previous paragraph, we can show that the restriction of $G$ to $\itrz(T_2 \cap T_3)$ also contains the restriction of $\Aut(\obs_{\w_1\cup\w_2})$ to $\itrz(T_2 \cap T_3)$. Indeed, for two distinct points $x, y \in \itrz(T_2 \cap T_3)$ in the same fiber of $\Aut(\obs_{\w_1\cup\w_2})$, we can find $\sigma_x, \sigma_y$ in the $w_2$-subgroup of $G$ that restrict to transpositions on $\itrz(T_2)$ sending respectively $x$ and $y$ in $\itrz(T_1 \cap T_2)$. Again, this is possible since the restriction $\obs_{\w_2}$ to $\itrz(T_1 \cap T_2)$ is surjective. Now, by the above result,  we can find $\tau\in G$ in the restriction of $\Aut(\obs_{\w_1\cup\w_2})$ to $\itrz(T_1 \cap T_2)$ that restricts to the transposition sending $\sigma_x(x)$ to $\sigma_y(y)$. It follows that $(\sigma_x\circ\sigma_y)\circ \tau \circ (\sigma_x\circ\sigma_y)^{-1}\in G$ fixes $\itrz(T_2 \cap T_3)$ and restricts to the transposition $(x,y)$. Now that we have proven that the restriction of $G$ to $\itrz(T_2 \cap T_3)$ also contains the restriction of $\Aut(\obs_{\w_1\cup\w_2})$ to $\itrz(T_2 \cap T_3)$, it is clear that we can repeat the same arguments to prove that the restriction of $G$ to $\itrz(T_3 \cap T_4)$ contains the restriction of $\Aut(\obs_{\w_1\cup\w_2\cup \w_3})$ to $\itrz(T_3 \cap T_4)$ and so on. By running this induction twice around $\Z/n\Z$, it implies that the restriction of $G$ to $\itrz(T_j \cap T_{j+1})$ contains the restriction of $\Aut(\obs_{\cup\w_j})$ to $\itrz(T_j \cap T_{j+1})$. Since $\cup\w_j=\partial \Delta$, the claim follows.
\end{proof}

\begin{proof}[Proof of Proposition \ref{prop:2}]
Let $j_0 \in \Z/n\Z$ be such that $q= q_{j_0}$.  Under the present assumptions, we claim that for any fiber of $\obs_{\partial \Delta}$, there exists a transposition $\tau\in G$ with support in $\itrz(T_{j_0})$. Then, for any distinct pair $x,y \in \itrz(T_{j_0})$ in the same fiber of $\obs_{\partial \Delta}$, the group $G$ contains the transposition $(x,y)$. Indeed, for a transposition $\tau \in G$ supported on the same fiber as $x$ and $y$, Proposition \ref{prop:1} ensures that there exist $\sigma_x, \sigma_y \in G$ that restrict to transpositions in $\itrz(T_{j_0})$ such that $\tau=(\sigma_x(x),\sigma_y(y))$. Therefore, the element $(\sigma_x\circ\sigma_y)\circ \tau \circ (\sigma_x\circ\sigma_y)^{-1}\in G$ is the transposition $(x,y)$. In particular, the group $G$ contains a strict $(\w_{j_0},\partial \Delta)$-group. The arguments of the previous proof imply that $G$ contains a strict $(\w_{j},\partial \Delta)$-group for any $j\in \Z/n\Z$ and the result follows.

It remains to prove the claim. To do so, choose coordinate of $M$ such that $\Delta_{j}$ is the segment joining $(0,0)$ to $(\ell,0)$ and such that the lattice point on $\Delta_{j-1}$ adjacent to $(0,0)$ has coordinates $(p,q)$ with $0\leq p<q$ (with $p=0$ if and only if $q=1$). 

If $q=1$, consider the wedge $\w=\big\lbrace(0,0), \, (1,0), \, (2,0), \, (3,0), \, (0,4) \big\rbrace$. Since $\partial \conv(\w) \cap \itrz(\Delta) = \varnothing$ and $\itrz\big(\conv(\w)\big)$ contains $(1,1)$ and $(2,1)$, any $\w$-group (and therefore $G$) contains the transposition of the latter points. Since $\obs_{\partial\Delta}$ has only one fiber when $q=1$, the claim is proven in this case.

If $q>1$, consider the wedge $\w=\big\lbrace(0,0), \, (1,0),..., \, (3p+1,0), \, (3p,3q) \big\rbrace$. Once again, we have that $\partial \conv(\w) \cap \itrz(\Delta) = \varnothing$ since the vector $(3p,3q)-(3p+1,0)=(-1,3q)$ is primitive. Observe also that the horizontal section of $\itrz\big(\conv(\w)\big)$ at height $m$ for any $1\leqslant m \leqslant \lfloor q/2 \rfloor$ has length at least $5p/2\geqslant 5/2$ and therefore contains at least two lattice points. It follows that any $\w$-group (and therefore $G$) contains the transposition of the latter points. Since $\obs_{\partial\Delta}(u,v)=\dist(v,q'\Z)$ for some divisor $q'$ of $q$, it follows that any fiber of $\obs_{\partial\Delta}$ contains the support of at least one of the above transpositions.
\end{proof}

\begin{proof}[Proof of Proposition \ref{prop:3}]
For any $x \in \itrz(\Delta)$, we claim that there exists $j\in \Z/n\Z$ such that $G$ contains a permutation $\sigma_x$ sending $x$ to $\itrz(T_j\cap T_{j+1})$ and such that $\sigma_x(x)$ is the only element of the support of $\sigma_x$ in $\itrz(T_j)$. Then, Proposition \ref{prop:2} and \textbf{(A)} imply the existence of a transposition $\tau \in G$ supported on $\itrz(T_j\cap T_{j+1})$ that contains $\sigma_x(x)$. In turn, it implies the existence of a transposition in $G$, namely $\sigma_x\circ \tau\circ \sigma_x^{-1}$ that maps $x$ into $\itrz(T_j\cap T_{j+1})$. Conjugating by another transposition in the strict $(\w_j,\partial \Delta)$-subgroup of $G$ if necessary, we can construct a transposition in $G$ that maps $x$ into $\itrz(T_j \cap T_{j-1})$ and inductively to any $\itrz(T_{j'})$. From there, it is clear by Proposition \ref{prop:2} that $G$ contains the transposition $(x,y)$ for any distinct $x,y \in \itrz(\Delta)$ in the same fiber of $\obs_{\partial \Delta}$.

It remains to prove the claim. Below, we assume that $n\geqslant 4$ since the case $n=3$ is trivial. In order to prove the claim, observe that there exist a wedge $\w(j,v)$ such that  $x \in \itrz\big(T(j,v)\big)$ and such that $v$ is outside of the relative interiors of $\Delta_{j-1}$ and $\Delta_{j+1}$. Indeed, there clearly exist indices $j, \, k \in \Z/n\Z$ such that $x \in \itrz\big(\conv(\Delta_j \cup \Delta_k)\big)$. If $\vert k-j\vert=1$, we can take $v$ to be the vertex $\Delta_k\cap\Delta_j$ and we are done. If not, the set $\conv(\Delta_j \cup \Delta_k)$ is a quadrilateral. If $x$ is not the intersection point of the diagonals of $\conv(\Delta_j \cup \Delta_k)$, it lies in the interior of one of the four triangles defined by the diagonals and we are done. Otherwise, take any $v \in \Delta_k \cap M$ other than a vertex.

Let $\w^0:= \w(j,v)$ be such that $x \in \itrz\big(\conv(\w^0)\big)$. If $\w^0$ is either $T_j$ or $T_{j+1}$, we are done by \textbf{(A)}. If not, then there exists $v' \in (\partial \Delta \cap M)\setminus \Delta_j$ that is consecutive to $v$ on $\partial \Delta \cap M$. In particular, the wedge $\w^1:=\w(j,v')$ is well defined. Up to a change of coordinates, we can assume without loss of generality that $v'$ is consecutive to $v$ when running along $\partial \Delta$ clockwise. By \textbf{(C)}, 
we can find in any $\w^0$-group (and therefore in $G$) a permutation $g_1$ whose restriction to $\itrz(\Delta) \setminus \partial \conv(\w^0)$ is a transposition sending $x$ in $\itrz\big(\conv(\w^1)\big)$. 
Observe that the support of $g_1$ is disjoint from $\itrz\big(T_j\cap T_{j+1}\big)$, except possibly for $g_1(x)$. Indeed, for any wedge $\w:=\w(j,v'')$, the boundary $\partial \conv(\w)$ is disjoint from $\itrz\big(T_j\cap T_{j+1}\big)$. If $g_1(x)$ is in $\itrz\big(T_j\cap T_{j+1}\big)$, then we are done. Otherwise, we apply the same procedure to $x_1:=g_1(x)$ as a point in $\itrz\big(\conv(\w_1)\big)$ and obtain a point $x_2:=g_2(g_1(x))$ in $\w^2$, and so on.

After finitely many steps $m$, the wedge $\w^m$ is $\w(j, \Delta_{j+1,j+2})$ and we constructed $g_1,\cdots,g_m\in G$ such that $x_m:=g_m\circ \cdots\circ g_1(x)$ lies in $\itrz\big(\conv(w^n)\big)= \itrz(T_{j+1})$ and such that the support of $g_m\circ \cdots\circ g_1$ is disjoint from $\itrz\big(T_j\cap T_{j+1}\big)$, except possibly for $x_m$. If $x_m$ lies in the latter, we are done. Otherwise, By Proposition \ref{prop:2}, there exists a transposition $g_{n+1}\in G$ sending $x_m$ in $\itrz\big(T_j\cap T_{j+1}\big)$. The sought permutation is then the product $g_{m+1}\circ \cdots\circ g_1$.
\end{proof}

In order to prove Proposition \ref{prop:l}, we will need the following elementary fact. 

\begin{Lemma}\label{lem:quadri}
$\mathbf{1.}$ Let $A,B,C,D \in \R^2$ such that $ABCD$ is a convex quadrilateral and denote $O:=AC \cap BD$ and $\theta_A$ and $\theta_B$ the positive angles in $ABCD$ at the vertices $A$ and $B$. Then, we have 
\[ \dfrac{AO}{AC} = \frac{AB \cdot AD\cdot \sin(\theta_A)}{AB \cdot AD\cdot \sin(\theta_A)+AB \cdot BC \cdot \sin(\theta_B)-AD \cdot BC\cdot \sin(\theta_A+\theta_B)}.\]
$\mathbf{2.}$ Assume moreover that either $AB$ and $CD$ are parallel or the lines $(AB)$ and $(CD)$ intersect at a point $P$ such that $A \in PB$. If $\theta$ denotes the angle $\overbow{APD}$ (set $\theta=0$ in the parallel case), then we have 
\[ \dfrac{AO}{AC} = \frac{AB \cdot AD\cdot \sin(\theta_A)}{AB \cdot AD\cdot \sin(\theta_A)+ CD \cdot \big( AD \sin(\theta_A- \theta)+AB \sin(\theta)\big)}.\]
\end{Lemma}

\begin{proof}
$\mathbf{1.}$ Let $\theta'$ be the positive angle at $A$ in $ABO$. The sought formula follows from 
\[AO=\frac{AB \cdot AD\cdot \sin(\theta_A)}{AD\cdot \sin(\theta_A-\theta')+AB \cdot \sin(\theta')}, \; \; \; AC=\frac{\sin(\theta_B)}{\sin(\theta')}BC\]
and the identity
\[ \sin(\theta_A+\theta_B)\sin(\theta')+\sin(\theta_A-\theta')\sin(\theta_B)=\sin(\theta_A)\sin(\theta_B+\theta').\]
The latter identity follows from the expansion of $\sin(a+b)$ and the second formula follows from the law of sines. For the first formula, we use the relation $$\area(ABD)=\area(AOD)+\area(ABO)$$ and compute that $\area(ABD)=\frac{1}{2}AB\cdot AD\sin(\theta_A)$, that   $\area(AOD)=\frac{1}{2}AO\cdot AD\sin(\theta_A-\theta')$ and that $\area(ABO)=\frac{1}{2}AB\cdot AO\sin(\theta')$.
%

$\mathbf{2.}$ Let $B'$ be the point on $(BD)$ such that $AB'$ is parallel to $CD$. Under our assumptions, the point $B'$ lies on $BD$. In the trapezoid $AB'CD$, we have
$\frac{AO}{AC}=\frac{AB'}{AB'+CD}$. As above, we compute that 
\[AB'=\frac{AB \cdot AD\cdot \sin(\theta_A)}{AD\cdot \sin(\theta_A-\theta)+AB \cdot \sin(\theta)}\]
from which we deduce the sought formula.
\end{proof}

We are now ready to prove Proposition \ref{prop:l}. In this context, we choose coordinates $M\simeq \Z^2$. Recall that for any $j \in \Z/n\Z$, the index $q_j$ of the affine lattice generated by $v_{j-1}$ and $v_j$ in the lattice $M\simeq \Z^2$ is given by $q_j=\det(v_{j-1}, v_j)= \vert v_{j-1} \vert \cdot \vert v_j \vert \cdot \sin(\theta)$, where $\theta$ is the positive angle between $v_{j}$ and $v_{j-1}$. In the course of the proof below, we will use the following minoration
\begin{equation} \label{eq:minoration}
\vert \Delta_{j-1}\vert \cdot \vert \Delta_j \vert \cdot \sin(\theta)=q_j \cdot \length(\Delta_{j-1}) \cdot \length(\Delta_{j}) \geqslant \length(\Delta_{j-1}) \cdot \length(\Delta_{j}).
\end{equation}
We also warn that we use the same notation $AB$ to denote the segment between two points $A, B \in \R^2$ and to denote the Euclidean  distance between them.

\begin{proof}[Proof of Proposition \ref{prop:l}]
In the context of $\mathbf{(B)}$, observe that $q_j\leqslant \max_{j\in \Z/n\Z}\left\lbrace \right\vert v_j\vert^2 \rbrace$ and in turn that $\ell \geqslant 5q$. In particular, the assumption $\mathbf{(B)}$ is satisfied.

Now, consider two wedges $\w:=\w(j,v)$ and $\w':=\w(j,v')$ as in $\mathbf{(C)}$. Label the vertices of the convex quadrilateral $\conv( \w \cup \w')$ by $A,B,C,D \in M\simeq\Z^2$ such that $AB= \Delta_j$ and such that $A,B,C,D$ satisfy the assumptions of Lemma \ref{lem:quadri}.$2$ (let us say that $C=v$). In particular, $C$ and $D$ are consecutive lattice points on some edge $\Delta_k$. In order to show that the restrictions of $\obs_{\w}$ and $\obs_{\w'}$ to $\itrz\big(T(j,v)\cap T(j,v')\big)$ are surjective, it suffices to show that $\frac{AO}{AC}>\frac{1}{2}$ and that the section of $ABO$ parallel to $AB$ and passing through the midpoint of $AC$ has Euclidean length strictly greater than $\vert v_j \vert $. 
Indeed, if we change coordinates on $M$ such that $AB$ is on the first coordinate axis (so that $\vert v_j \vert $ becomes $1$) and denote by $A'$ and $B'$ the intersection of the latter section with $AO$ and $BO$, then the trapezoid $AA'B'B$ contains at least one lattice point in the interior of each horizontal section at integer height. In these coordinates, the maps $\obs_{\w}$ and $\obs_{\w'}$ are given by $\obs_{\w}(u,v)=\dist(v,q\Z)$ and $\obs_{\w'}(u,v)=\dist(v,q'\Z)$ with $q\geqslant q'$ and the second coordinates of $A'$ and $B'$ are equal to $q/2 \geqslant q'/2$. Therefore, the restrictions of $\obs_{\w}$ and $\obs_{\w'}$ to $\itrz(ABO)$ are surjective.

Let us now show that $\frac{AO}{AC}>\frac{1}{2}$ and $A'B'> \vert v_j \vert $. We will proceed by deriving successive lower bounds for $\frac{AO}{AC}$ starting from the formula given in Lemma \ref{lem:quadri}.$2$. As a preparation, let us show that $AD \geqslant 5 \cdot \vert v_k \vert$. Indeed, if we assume that $A$ is the vertex $\Delta_{j-1,j}$ (the case $A = \Delta_{j,j+1}$ is similar), then the distance $AD$ is greater than the distance from the line $(AB)$ to the vertex $\Delta_{j-2,j-1}$. The latter distance is equal to $\det(v_j, v_{j-1}) \cdot \length(\Delta_{j-1})\cdot \vert v_j \vert^{-1}$. Since $\length(\Delta_{j-1}) \geqslant 5 \cdot \max_{i\in \Z/n\Z}\left\lbrace \vert v_i \vert^2 \right\rbrace$, the claim follows. Recall that, in the terminology of Lemma \ref{lem:quadri}.$2$, we have that $\sin(\theta) < \sin(\theta_A)$, that $CD= \vert v_k \vert$ and 
that $AB\geqslant 5 \vert v_k \vert ^2 \geqslant 5 \vert v_k \vert$.

Starting from the formula in Lemma \ref{lem:quadri}.$2$, we deduce that
\[
\begin{array}{rl}
\dfrac{AO}{AC} 	& = \dfrac{AB \cdot AD\cdot \sin(\theta_A)}{AB \cdot AD\cdot 												\sin(\theta_A)+ CD \cdot \big( AD \sin(\theta_A- \theta)+AB 												\sin(\theta)\big)}\\
\\
							& > \dfrac{AB \cdot AD\cdot \sin(\theta_A)}{AB \cdot AD\cdot 											\sin(\theta_A)+ CD \cdot \big( AD \big(\sin(\theta_A)+\sin(\theta)										\big) +AB \sin(\theta)\big)}\\
\\							
							& = \dfrac{AB\cdot AD}{AB\cdot AD+  AD\cdot CD \cdot \left( 									1+\frac{\sin(\theta)}	{\sin(\theta_A)}\right)+ AB \cdot CD \cdot 											\frac{\sin(\theta)}{\sin(\theta_A)}}\\
							\\
							& > \dfrac{AB\cdot AD}{AB\cdot AD+  2 \, AD\cdot CD \cdot+ AB 										\cdot CD} = \dfrac{1}{1+  2 \frac{\vert v_k \vert}{AB} + \frac{\vert v_k \vert}									{AD}}
							\geqslant\dfrac{1}{1+\frac{2}{5} + \frac{1}{5}}
							=\dfrac{1}{2} \cdot \dfrac{5}{4}.
\end{array}
\]
In turn, we have $A'B'= AB \cdot \big(1- \frac{AA'}{AO}\big)= AB \cdot \left(1- \frac{1}{2}\cdot\frac{AC}{AO}\right) > AB \cdot \left(1- \frac{4}{5}\right) = \frac{AB}{5} \geqslant \frac{\ell \cdot \vert v_j\vert}{5} \geqslant \vert v_j\vert$.
We conclude that $\mathbf{(C)}$ is satisfied.

Let us now consider the pair of wedges $\w_j$ and $\w_{j+1}$ in the context of $\mathbf{(A)}$. Label the vertices of the convex quadrilateral $\conv( \w_j \cup \w_{j+1})$ by $A,B,C,D \in \Z^2$ such that $AB= \Delta_j$ and such that $A,B,C,D$ satisfy the assumptions of Lemma \ref{lem:quadri}.$2$.
Define $A' \in AC$ and $B' \in BD$ such that $A'B'$ is parallel to $AB$ and $\frac{AA'}{AC}=\frac{1}
{\length(BC)}$. In order to show that the restrictions of $\obs_{\w_j}$ and $\obs_{\w_{j+1}}$ to $
\itrz(ABO)$ are at least $2$-to-$1$, it suffices to show that $\frac{AO}{AC}>\frac{1}{\length(BC)}$ 
and that $A'B'> \vert v_j \vert $. Indeed, the trapezoid $AA'B'B$ (and therefore $ABO$) will contains 
two translated parallelograms spanned by $v_{j-1}$, $v_j$ and $v_j$, $v_{j+1}$ respectively. 
Using similar arguments as above, this leads to the sought statement. 

Starting from Lemma \ref{lem:quadri}.$1$, we deduce that
\[
\begin{array}{rl}
\length(BC) \cdot \dfrac{AO}{AC} 	
		& = \dfrac{ \length(BC) \cdot AB \cdot AD\cdot \sin(\theta_A)}{AB \cdot 				AD\cdot \sin(\theta_A)+AB \cdot BC \cdot \sin(\theta_B)-AD \cdot BC\cdot 					\sin(\theta_A+	\theta_B)}\\
\\
& > \dfrac{  \length(BC) \cdot \length(AB) \cdot \length(AD)}{AB \cdot AD +AB \cdot BC \cdot +AD \cdot BC} \quad \quad \big(\text{using } \eqref{eq:minoration} \text{ on the numerator}\big)\\
\\
& \geqslant \dfrac{  \length(BC) \cdot \length(AB) \cdot \length(AD)}{\big(\length(AB) \cdot \length(AD) +\length(AB) \cdot \length(BC)  \cdot +\length(BC) \cdot \length(AD)\big) \cdot \max_{ i \in \Z/n\Z} \left\lbrace \vert v_i \vert^2  \right\rbrace}\\
\\
& = \dfrac{ 1 }{\left(\frac{1}{\length(BC)}  +\frac{1}{\length(AD)} +\frac{1}{\length(AB)}\right) \cdot \max_{ i \in \Z/n\Z} \left\lbrace \vert v_i \vert^2  \right\rbrace}
\geqslant \dfrac{1}{\frac{3}{\ell} \cdot \max_{ i \in \Z/n\Z} \left\lbrace \vert v_i \vert^2  \right\rbrace}
\geqslant \dfrac{5}{3}.
\end{array}
\]
In turn, we have $A'B= AB \cdot \big(1- \frac{AA'}{AO}\big)= AB \cdot \left(1- \frac{1}{\length(BC)}\cdot\frac{AC}{AO}\right) > AB \cdot \left(1- \frac{3}{5}\right) = \frac{2}{5} AB \geqslant \frac{2 \cdot \ell}{5} \vert v_j\vert > \vert v_j\vert$.
We conclude that $\mathbf{(A)}$ is satisfied.
\end{proof}

\section{Pairs $(X,\Li)$ with unexpected monodromy}\label{sec:counterex}

We now restrict our attention to pairs $(X,\Li)$ induced by lattice polygons $\Delta\subset \R^2$ of the following form:
the set $\Delta \cap \Z^2$ is contained on a line $L$ except for two points $p$ and $q$ lying on different sides of $L$. Assume further that the segment connecting $p$ to $q$ contains a lattice point of $L$. Using an affine linear transformation if necessary, we can assume that $L$ is the second coordinate axis of $\Z^2$ and that $(\Delta \setminus L) \cap \Z^2= \big\lbrace (-1,0), \, (1,0) \big\rbrace$.
In particular, the polygon $\Delta$ is either a quadrilateral or a triangle with at least one symmetry. We will refer to such a polygon as a \emph{kite}. 

Then, any polynomial $f \in H^0(X,\Li)$ is of the form
\[f(z,w)= \frac{\alpha}{z}+p(w)+\beta z\]
where $p(w)$ is a Laurent polynomial and $\alpha,\beta\in \C$.

In this setting, the singular points of any curve $C \in \LL$ carry a particular information. Indeed, a curve $C$ (whose very affine part $\Cb$ is defined by a Laurent polynomial $f$) has a singular point if the system 
$f=\partial_z f=\partial_w f=0$
has a solution in $\alpha$, $\beta$ and the coefficients of $p$. From the equation $\partial_z f=0$, we deduce that the coordinates $(z,w)$ of a node satisfy the relation
\[-\alpha +\beta z^2=0 \; \Leftrightarrow \; z^2=\frac{\alpha}{\beta} \; \Leftrightarrow \; \textstyle\big(z- \sqrt{\frac{\alpha}{\beta}}\big)\big(z+ \sqrt{\frac{\alpha}{\beta}}\big)=0\]
for any determination of $\sqrt{\frac{\alpha}{\beta}}$,
that is we can decorate any node of $C$ with one of the two square roots of $\frac{\alpha}{\beta}$ and this decoration is continuous in the coefficients of $f$.

\begin{Lemma}
Let $C\in \sv{\Delta}$ be a rational simple Harnack curve. Then, the above decoration defines a surjective map $\partial : \{\text{nodes of } C\}\rightarrow \big\lbrace \pm\sqrt{\alpha/\beta}\big\rbrace$ with fibers of cardinality $\lfloor k/2 \rfloor$ and $\lceil k/2 \rceil$ for $k\geqslant 2$. In particular, the image of the map $\mu_\Delta$ is imprimitive for $k\geqslant 4$.
\end{Lemma}

Before proving this result, we would like to mention that the nature of the obstruction $\partial$ was discovered during a conversation with I. Tyomkin to whom the author is particularly grateful.

\begin{proof}
It follows from \cite[Lemma 11]{Mikh} that the node $\nu\in C$ such that $\ord(\nu)=(0,\kappa)$ sits in the quadrant of $(\R^*)^2$ of sign $\big(+1,(-1)^\kappa)\big)$, up to an appropriate sign change of the coordinates. Therefore, the fibers of $\partial$ are in bijective correspondence with the integers $1\leqslant \kappa \leqslant k$ with a given reduction modulo $2$. This proves the first part of the statement. For the second part, observe that $\pi_1(\sv{\Delta})$ acts on $\big\lbrace \pm\sqrt{\alpha/\beta}\big\rbrace$ by permutations and it is clear from the parametrisation \eqref{eq:param} that this action is surjective. In particular, for any $\gamma \in \pi_1(\sv{\Delta})$, the permutation $\mu_\Delta(\gamma)$ preserves the fibers of $\partial$ if the corresponding permutation of the roots of $\frac{\alpha}{\beta}$ is the identity and exchanges the fibers otherwise. In other words, the fibers of $\partial$ are blocks for the action of $\mu_\Delta$. This blocks are non-trivial provided that $k\geqslant 4$.
\end{proof}

\begin{Remark}
$\mathbf{1}.$
We can consider a kite $\Delta$ such that $\partial \Delta \cap \Z^2$ generates the full lattice. In such case, the obstruction map $\obs_{\partial \Delta}$ is trivial. However, as we have seen above, the map $\mu_{\Delta}$ is not surjective in general. In particular, it justifies the necessity of the constant $\ell$ in Theorem \ref{thm:main}.\\
$\mathbf{2}.$ In \cite{LT}, we show in particular that the Severi varieties of pairs $(X,\Li)$ obtained from kites are reducible in general.
\end{Remark}

\bibliographystyle{alpha}
\bibliography{Draft}

\vspace{1cm}
\noindent
L. Lang\\
Department of Mathematics, Stockholm University, SE - 106 91 Stockholm, Sweden.\\
\textit{Email}: lang@math.su.se

\end{document}